\def\UseSection{
      \numberwithin{equation}{section}
	\theoremstyle{plain}
      \newtheorem{theorem}    {Theorem}[section]
      \DefineTheorems 
}

\def\DefineTheorems{
	
	\newtheorem{lemma}      [theorem] {Lemma}
	
	\newtheorem{prop}       [theorem] {Proposition}
	
	\newtheorem{cor}        [theorem] {Corollary}

	\theoremstyle{definition}
	\newtheorem{defn}       [theorem] {Definition}

	\theoremstyle{definition}

}

\newcommand{\bt}   {\begin{theorem}}
\newcommand{\et}   {\end  {theorem}}
\newcommand{\bl}   {\begin{lemma}}
\newcommand{\el}   {\end  {lemma}}
\newcommand{\bp}   {\begin{prop}}
\newcommand{\ep}   {\end  {prop}}
\newcommand{\bc}   {\begin{cor}}
\newcommand{\ec}   {\end  {cor}}
\newcommand{\bd}   {\begin{defn}}
\newcommand{\ed}   {\end  {defn}}

\newcommand{\ba}   {\begin{array}}
\newcommand{\ea}   {\end  {array}}
\newcommand{\be}   {\begin{enumerate}}
\newcommand{\ee}   {\end  {enumerate}}
\newcommand{\bi}   {\begin{itemize}}
\newcommand{\ei}   {\end  {itemize}}

\def\eq#1\en{\begin{equation}#1\end{equation}}  
\def\eqsplit#1\ensplit{
	\begin{equation}\begin{split}#1\end{split}\end{equation}
	}
\def\eqalign#1\enalign{
	\begin{align}#1\end{align}
	}
\def\eqmul#1\enmul{
	\begin{multline}#1\end{multline}
	}
\newcommand{\eqarrstar} {\begin{eqnarray*}} 
\newcommand{\enarrstar} {\end{eqnarray*}} 
\newcommand{\eqarray}   {\begin{eqnarray}} 
\newcommand{\enarray}   {\end{eqnarray}}

\newcommand{\lbeq}[1]  {\label{e:#1}}
\newcommand{\refeq}[1] {\eqref{e:#1}}    

\makeatletter
\newcommand{\labelcounter}[2]{{%
	\stepcounter{#1}
	\protected@write\@auxout{}%
	{\string\newlabel{#2}{{\csname the#1\endcsname}{\thepage}}}%
	{\ref{#2}}
	}}
\makeatother
\newcommand{\Nbold} {{\mathbb N}}

\newcommand{\Zbold} {{\mathbb Z}}

\newcommand{\spose}[1] {{\hbox to 0pt{#1\hss}} }
\newcommand{\ltapprox} {\mathrel{\spose{\lower 3pt\hbox{$\mathchar"218$}}
\raise 2.0pt\hbox{$\mathchar"13C$}}}
\newcommand{\gtapprox} {\mathrel{\spose{\lower 3pt\hbox{$\mathchar"218$}}
\raise 2.0pt\hbox{$\mathchar"13E$}}}

\documentclass[12pt]{article}
\usepackage[psamsfonts]{amsfonts}
\usepackage{amsmath,amssymb,amsthm}
\usepackage[dvips]{graphicx}
\usepackage{eepic}
\newtheorem{THM}{Theorem}[section]
\newtheorem{COR}[THM]{Corollary}
\newtheorem{EXA}[THM]{Example}

\newtheorem{LEM}[THM]{Lemma}
\newtheorem{PRP}[THM]{Proposition}
\newtheorem{DEF}[THM]{Definition}

\newcommand{\nn}{\nonumber}

\UseSection  
\newcommand{\hlf}{\frac{1}{2}}
\newcommand{\ra}{\rightarrow}
\newcommand{\lra}{\leftrightarrow}

\renewcommand{\to}      {\rightarrow}

\newcounter{countC}  
\newcounter{countR}  
\newcommand{\Z}{\Zbold}
\newcommand{\N}{\Nbold}

\newcommand{\floor}[1]{\lfloor #1 \rfloor}

\newcommand{\mc}[1]{\mathcal{#1}}

\newcommand{\mE}{\mathbb{E}}

\newcommand{\UD}{\updownarrow}
\newcommand{\LR}{\leftrightarrow}
\newcommand{\WE}{\LR}
\newcommand{\NS}{\UD}

\newcommand{\smallE}{\scriptstyle \rightarrow}
\newcommand{\smallW}{\scriptstyle \leftarrow}
\newcommand{\ssmallW}{\scriptscriptstyle \leftarrow}
\newcommand{\smallN}{\scriptstyle \uparrow}
\newcommand{\smallS}{\scriptstyle \downarrow}

\newcommand{\smallNS}{\scriptstyle \updownarrow}
\newcommand{\smallnw}{\scriptscriptstyle \nwarrow}
\newcommand{\smallsw}{\scriptscriptstyle \swarrow}

\newcommand{\NE}{\begin{picture}(,)
\put(2,-5){$\rightarrow$}
\put(0,.5){$\uparrow$}
\end{picture}\hspace{.5cm}
}
\newcommand{\smallNE}{\begin{picture}(,)
\put(1.5,-3){$\smallE$}
\put(0,.5){$\smallN$}
\end{picture}\hspace{.35cm}
}
\newcommand{\SE}{\begin{picture}(,)
\put(1.5,4.8){$\rightarrow$}
\put(-.5,-.5){$\downarrow$}
\end{picture}\hspace{.5cm}
}
\newcommand{\smallSE}{\begin{picture}(,)
\put(1.4,3.2){$\smallE$}
\put(,-.5){$\smallS$}
\end{picture}\hspace{.35cm}
}
\newcommand{\SW}{\begin{picture}(,)
\put(0,4.8){$\leftarrow$}
\put(8.2,-0.5){$\downarrow$}
\end{picture}\hspace{.5cm}
}

\newcommand{\NW}{\begin{picture}(,)
\put(0,-5){$\leftarrow$}
\put(8,0){$\uparrow$}
\end{picture}\hspace{.5cm}
}
\newcommand{\smallNW}{\begin{picture}(,)
\put(0.5,-3){$\smallW$}
\put(6.3,.5){$\smallN$}
\end{picture}\hspace{.35cm}
}

\newcommand{\NSE}{\begin{picture}(,)
\put(0,0){$\updownarrow$}
\put(2.2,0){$\rightarrow$}
\end{picture}\hspace{.5cm}
}
\newcommand{\SWE}{\begin{picture}(,)
\put(0,5){$\leftarrow$}
\put(5,5){$\rightarrow$}
\put(5.5,-0.5){$\downarrow$}
\end{picture}\hspace{.6cm}
}
\newcommand{\smallSWE}{\begin{picture}(,)
\put(0.5,2.5){$\smallW$}
\put(3.5,2.5){$\smallE$}
\put(3.8,-1.5){$\smallS$}
\end{picture}\hspace{.45cm}
}
\newcommand{\NWE}{\begin{picture}(,)
\put(0,-5){$\leftarrow$}
\put(5,-5){$\rightarrow$}
\put(5.5,0.5){$\uparrow$}
\end{picture}\hspace{.6cm}
}

\newcommand{\NSEWalt}{\begin{picture}(,)
\put(8.5,0){$\updownarrow$}
\put(0.5,0){$\longleftrightarrow$}
\end{picture}\hspace{.79cm}
}
\newcommand{\smallNSEWalt}{\begin{picture}(,)
\put(4.95,0){$\smallNS$}
\put(1,-1){$\leftrightarrow$}
\end{picture}\hspace{.45cm}
}
\newcommand{\smallOTSP}{\begin{picture}(,)
\put(6.2,0){$\smallN$}
\put(.3,-3.3){$\smallW$}
\put(1.7,0){$\smallnw$}
\end{picture}\hspace{.4cm}
}
\newcommand{\OTSP}{\begin{picture}(,)
\put(0.5,-5){$\leftarrow$}
\put(0.5,.2){$\nwarrow$}
\put(8.7,0.5){$\uparrow$}
\end{picture}\hspace{.5cm}
}
\newcommand{\smallFSOSP}{\begin{picture}(,)
\put(5.5,2.7){$\smallN$}
\put(5.5,-2.3){$\smallS$}
\put(.9,3.1){$\smallnw$}
\put(.9,-2){$\smallsw$}
\put(0.5,.5){$\ssmallW$}
\end{picture}\hspace{.35cm}
}
\newcommand{\FSOSP}{\begin{picture}(,)
\put(6.8,5){$\smallN$}
\put(6.8,-1){$\smallS$}
\put(2.5,5.1){$\smallnw$}
\put(2.5,0){$\smallsw$}
\put(1,2){$\smallW$}
\end{picture}\hspace{.4cm}
}

\newcommand{\blank}[1]{}

\usepackage[usenames]{color}

\newcommand{\ffG}{\mc{G}^{**}}

\title  {
      Degenerate random environments.
      }

\author{
Mark Holmes\footnote{Department of Statistics, University of Auckland.  E-mail {\tt
holmes@stat.auckland.ac.nz}} \and Thomas S. Salisbury \footnote{Department of Mathematics and Statistics, York University. E-mail {\tt
salt@yorku.ca}; Salisbury's research is supported in part by NSERC. Part of this work was carried out during a visit to the University of Auckland, whose hospitality he gratefully acknowledges. Both authors thank the referees for very thorough and helpful comments.}}

\begin{document}

\maketitle

\begin{abstract}
We consider connectivity properties of certain i.i.d.~random environments on $\Z^d$, where at each location some steps may not be available.  Site percolation and oriented percolation are examples of such environments.  In these models, one of the quantities most often studied is the (random) set of vertices that can be reached from the origin by following a connected path.  More generally, for the models we consider, multiple different types of connectivity are of interest, including: the set of vertices that can be reached from the origin; the set of vertices from which the origin can be reached; the intersection of the two.  As with percolation models, many of the models we consider admit, or are expected to admit phase transitions.  Among the main results of the paper is a proof of the existence of phase transitions for some two-dimensional models that are non-monotone in their underlying parameter, and an improved bound on the critical value for oriented site percolation on the triangular lattice.  
 The connectivity of the random directed graphs provides a foundation for understanding the asymptotic properties of random walks in these random environments, which we study in a second paper.
\end{abstract}

\section{Introduction}
\label{sec:general}
When studying random walks in environments that are non-elliptic (some nearest-neighbour steps may not be allowed from some locations), one should first consider the connectivity structure of the directed graphs that are induced by such environments.  In this paper, we introduce such random graphs in a general setting, with particular emphasis on models that connect to infinitely many sites almost surely.   We show that some such non-percolation models exhibit phase transitions,
 and use these results to improve existing bounds on the critical points for certain site-percolation models on the triangular lattice in 2 dimensions.  Many of the results of this paper are used in subsequent work where we study random walks in non-elliptic random environments \cite{HS_RWdRE}.

For fixed $d\ge 2$ let $\mc{E}_+=\{e_i: i=1,\dots,d\}$ be the set of standard basis vectors in $\Z^d$, and let $\mc{E}_-=\{-e_i: i=1,\dots,d\}$ and $\mc{E}=\mc{E}_+\cup \mc{E}_-$.  Let $\mc{P}$
denote the power set of $\mc{E}$.  For any set $A$, let $|A|$ denote the cardinality of $A$.  Let $\mu$ be a probability measure on $\mc{P}$.  For $A\in \mc{P}$ we will abuse notation and write $\mu(A)$ for $\mu(\{A\})$.  An i.i.d.~{\em degenerate random environment} is an element $\mc{G}=\{\mc{G}_x\}_{x\in \Z^d}$ of $\mc{P}^{\Z^d}$, equipped with the product $\sigma$-algebra and the product measure $\nu=\mu^{\otimes \Z^d}$.  We denote the expectation of a random variable $Z$ with respect to $\nu$ by $\mE[Z]$. 

We say that the environment is {\em $2$-valued} when $\mu$ charges exactly two points, i.e.~there exist distinct $A_1,A_2\in  \mc{P}$ and $p \in (0,1)$ such that $\mu(A_1)=p$ and $\mu(A_2)=1-p$.   In two dimensions, for fixed $A_1,A_2$  we will sometimes depict the corresponding family (indexed by $p=\mu(A_1)$) of models $(A_1\, A_2)$ pictorially.  Some well-known models fall within this framework.  
For example, setting $\mu(\mc{E})=p$ and $\mu(\emptyset)=1-p$, the random environment induced by $\mu$ is site percolation, and when $d=2$ we can depict it by $(\NSEWalt\, \cdot)$.  If instead we set $\mu(\mc{E}_+)=p$ and $\mu(\emptyset)=1-p$, we obtain oriented site percolation [in 2-dimensions $(\NE \,\cdot)$].  The model $(\uparrow\, \rightarrow)$ corresponds to a web of coalescing random walks, as in Arratia \cite{Arrat} or Toth-Werner \cite{TW}.  

\begin{figure}
\begin{center}
\includegraphics[scale=.45]{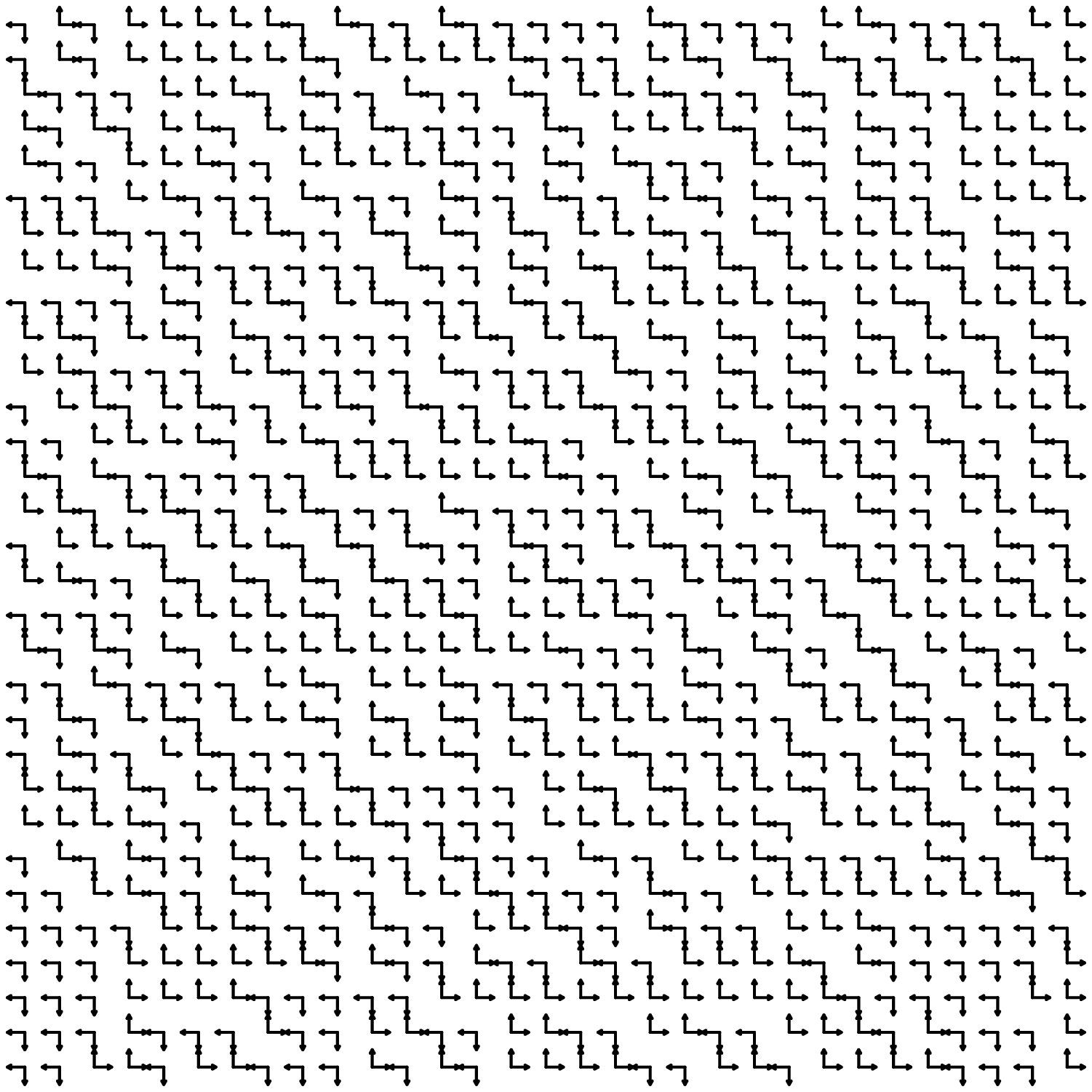} 
\includegraphics[scale=.45]{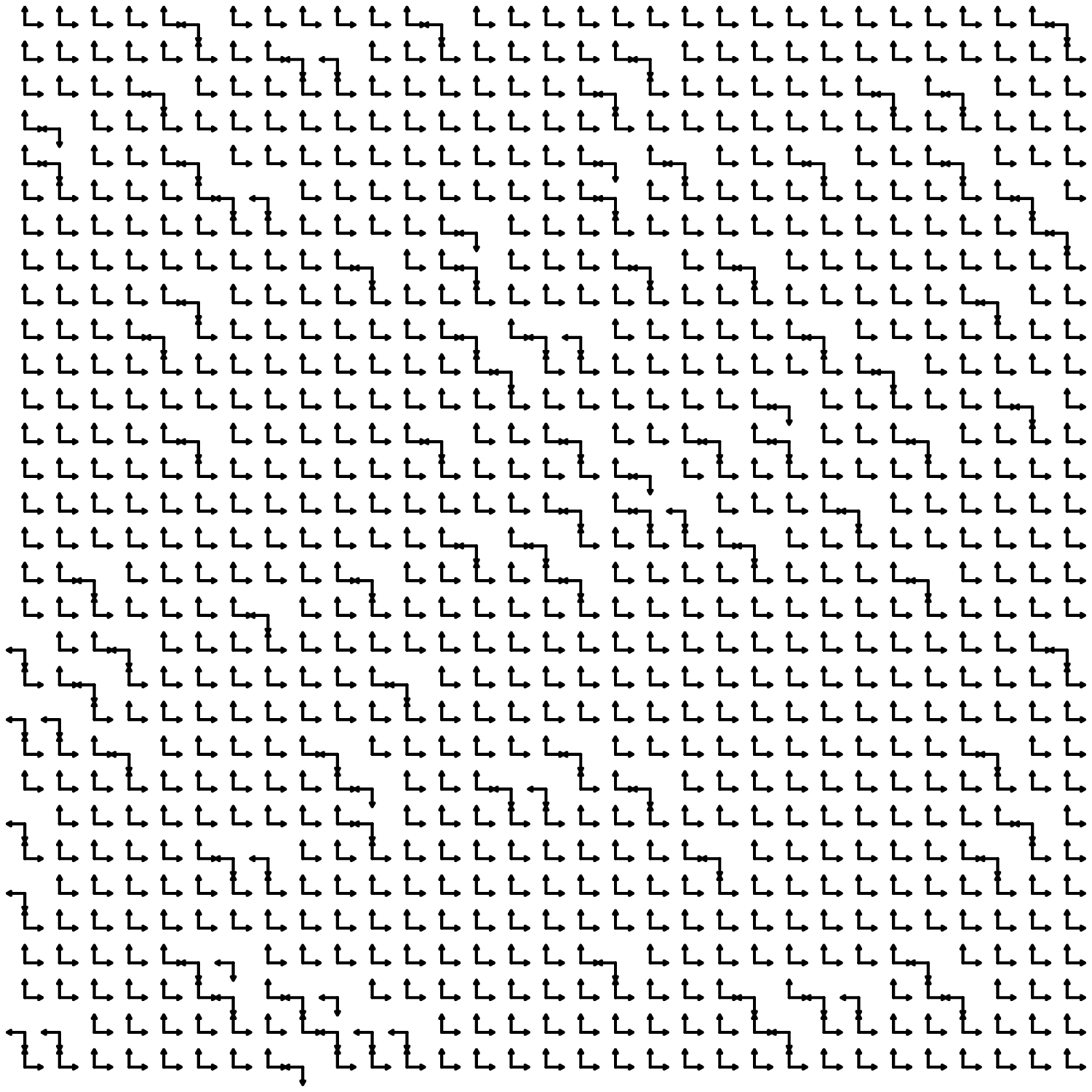} 
\end{center}
\vspace{-1cm}
\caption{Finite regions of the random environment in Example~\ref{exa:NE_SW} for $p=.5$ and $p=.9$ respectively.}
\label{fig:orthant_env}
\end{figure}

Three interesting 2-valued 2-dimensional examples $(A_1\,A_2)$ are the following.  
\begin{EXA} $(\NE\SW)$:
\label{exa:NE_SW}
i.e.~$A_1=\mc{E}_+$ and $A_2=\mc{E}_-$ (with $d=2$).  See Figure \ref{fig:orthant_env}.
 
\end{EXA}
\begin{EXA} $(\SWE\,\uparrow)$: 
\label{exa:SWE_N}
i.e.~$A_1=\{\rightarrow, \downarrow, \leftarrow\}$ and $A_2=\{\uparrow\}$.
\end{EXA}
\begin{EXA} $(\WE\,\NS)$:
\label{exa:NS_WE}
i.e.~$A_1=\{\rightarrow,\leftarrow\}$ and $A_2=\{\uparrow,\downarrow\}$.
\end{EXA}

Example \ref{exa:NE_SW} has superficial resemblances to corner percolation (see Pete \cite{Pete}), and to the Lorentz lattice gas model (see \S 13.3 of Grimmett \cite{Grim99}), though those models in fact seem unrelated.  Example \ref{exa:NS_WE} is a degenerate version of the ``good-node bad-node'' model of Lawler \cite{Lawler}.

While the examples that we find most interesting are 2-valued, there are of course other interesting models that lie within our framework.  One such example is oriented bond percolation \cite[Section 12.8]{Grim99}, where $p \in (0,1)$ and we define  
$\mu(B)=p^{|B|}(1-p)^{d-|B|}$ if $B\subset \mc{E}_+$ and $\mu(B)=0$ otherwise.

In $\Z^2$, re-label the four unit vectors in $\mc{E}$ as $f_1,\dots,f_4$. For $0<\lambda_1,\dots,\lambda_4<1$  take $\mu(B)=\prod_{i=1}^4\lambda_i^{1_B(f_i)}(1-\lambda_i)^{1_{B^c}(f_i)}$. Then each bond is randomly oriented in one or both directions (or is vacant), giving Grimmett's ``independent randomly oriented lattice'' model.  See Grimmett \cite{Grim02},  Wu and Zuo \cite{WuZuo}, and Linusson \cite{Lin}. This includes ``diode-resistor percolation'' as a special case. See Dhar et al \cite{DBP}, Redner \cite{Redner}, and Wierman \cite{Wierman85}.

The main results in this paper concern the structure of connected clusters $\mc{C}_x$, $\mc{B}_x$ and $\mc{M}_x$ in degenerate random environments, defined as follows.
\begin{DEF}
\label{def:connections}
Given an environment $\mc{G}$ and an $x=(x^{[1]},\dots,x^{[d]})\in \Z^d$, we say that:
\begin{itemize}
\item $x$ is {\em connected} to $y\in \Z^d$ and write $x\ra y$ if there exists an $n\ge 0$ and a sequence $x=x_0\,x_1,\dots, x_n=y$ such that $x_{i+1}-x_{i}\in \mc{G}_{x_i}$ for $i=0,\dots,n-1$;  
\item $x$ and $y$ {\em communicate} and write $x\lra y$, if $x\ra y$ and $y \ra x$;  
\item a nearest neighbour path in $\Z^d$ is {\em open in $\mc{G}$} if that path consists of edges in $\mc{G}$.
\end{itemize}
Let $\mc{C}_x=\{y\in \Z^d:x \ra y\}$, $\mc{B}_y=\{x\in\Z^d:x\ra y\}$, and $\mc{M}_x=\{y \in \Z^d:x\lra y\}=\mc{B}_x\cap\mc{C}_x$.
\end{DEF}
Three important quantities for this paper are the following probabilities
\[\theta_+=\nu(|\mc{C}_o|=\infty), \quad \theta_-=\nu(|\mc{B}_o|=\infty), \quad \text{and}\quad  \theta=\nu(|\mc{M}_o|=\infty).\]

For the model of Example \ref{exa:NE_SW}, we'll show that $\theta(p)=0$ for $p\notin (.16730,.83270)$, while for $p\in (.4534,.5466)$ there exists a unique infinite $\mc{M}$-cluster (and $\theta(p)>0$).  
There are two phase transitions in the model as $p$ varies -- with $\theta(p)$ moving from 0 to positive and back to 0.
The model is not monotone, in the sense that changing the local environment can both open and close connections. Nevertheless we'll relate the critical $p$'s to critical values for monotone percolation models. See Figures \ref{fig:orthant_env}, \ref{fig:ospC_o} and  \ref{fig:orthant_B_o}, Theorem \ref{thm:infiniteBorthant}, Corollary \ref{cor:Mfinite_2d_orthant}, and Theorem \ref{prp:infiniteM3}.

Similarly, the model of Example \ref{exa:SWE_N} is not monotone.  There is a unique infinite $\mc{M}$-cluster for $p>.4311$, while $\theta(p)=0$ for $p<.16730$. There is a phase transition, related to that of a monotone percolation model. 
See Figure \ref{fig:WSE_N_B_o}, Theorem \ref{thm:infiniteB5}, Corollary \ref{cor:Mfinite_2d_orthant}, and Theorem \ref{prp:infiniteM2}.


Example \ref{exa:NS_WE} is also non-monotone, but $\theta(p)>0$ for all $p\in (0,1)$ and there is almost surely a unique infinite $\mc{M}$-cluster.  See Theorem \ref{prp:infiniteM1}. Berger and Deuschel~\cite{BD11} prove a central limit theorem for this model.




\subsection{Main results}
\label{sec:results}
Since we study a whole class of models in this paper, there are both general and model-specific results. Many are short and elementary, while some are substantial.  

We use a broad range of classical methods that have been successful in studying percolation models, including blocking configurations, duality results and self-avoiding path counting arguments.  We cannot use the monotonicity property that is often used in percolation proofs either explicitly or implicitly e.g.~in establishing a sharp phase transition, or in proving the uniqueness of the infinite cluster, however 
we frequently exploit a tool that is not present in standard percolation models, namely the existence of  subnetworks of coalescing random walks (e.g.~open paths that use only steps in $\mc{E}_+$).


In addition to 
introducing a new and interesting class of random directed graph models, 
the following (2-dimensional) results are the highlights of this paper:
\begin{enumerate}
\item Proving a structure theorem (see Proposition \ref{PRP:trichotomy} and Corollary \ref{COR:trichotomy}), giving the possible forms of $\mc{B}_x$ in two dimensions, under fairly broad conditions.
\item Proving the existence of sharp phase transitions, both for infinite $\mc{B}_x$ clusters (e.g.~see Theorems  \ref{thm:infiniteBorthant} and \ref{thm:infiniteB5}) and for the existence of a gigantic $\mc{M}$ cluster (see Definition \ref{def:gigantic} and 
Theorem \ref{BvsM}).  
\item Improving existing rigorous bounds on the critical values of oriented site-percolation models on the triangular lattice. These follow from bounds on the critical values for our random directed graph models, together with a duality argument (see Theorems \ref{prp:infiniteM2} and \ref{prp:infiniteM3}).
\end{enumerate}

\section{The set of points $\mc{C}_x$ that can be reached from $x$}
\label{sec:C_x}
In this section we investigate properties of the random sets $\mc{C}_x\subset \Z^d$.  
In addition to standard percolation models, there are plenty of other models where $0<\theta_+<1$.  Consider for example a 4-valued model $(\NE\,\SE\,\SW\,\NW)$, where $\mu(\{\uparrow, \rightarrow\})>0$, $\mu(\{\downarrow, \rightarrow\})>0$, 
$\mu(\{\leftarrow, \downarrow\})>0$, $\mu(\{\uparrow, \leftarrow\})>0$. Then $\nu(|\mc{C}_o|=4)>0$ (see also Lemma \ref{lem:criteriontobeinfinite} below), but if any one of these local configurations occurs with probability greater than the critical value $p_c^{\smallNE}$ of oriented site percolation then clearly also $\theta_+>0$.




As for standard percolation models, $\{\exists x:|\mc{C}_x|=\infty\}$ is a tail event, giving the following result.
\begin{LEM}
\label{lem:exist_inf_C}
If $\theta_+>0$ then $\nu(\exists x:|\mc{C}_x|=\infty)=1$.
\end{LEM}

When studying degenerate random environments, and random walks therein, our principal interest will be in situations where the following condition (which prevents the random walk from getting stuck on a finite set of sites, see \cite{HS_RWdRE}) holds: 
\begin{equation}
\label{standinghypothesis}
\theta_+=1.
\end{equation}
The following is an explicit condition on $\mu$ that is equivalent to \eqref{standinghypothesis}.

\begin{LEM} Fix $d\ge 1$.  Then $\theta_+=1$ if and only if there exists a set $V$ of mutually orthogonal unit vectors such that $\mu(\{A:A\cap V\neq\emptyset\})=1$.
\label{lem:criteriontobeinfinite}
\end{LEM}
\proof If such a set $V$ exists then trivially we can construct an infinite self-avoiding path by always following a vector chosen from $V$.

For any $E\subset\{1,2,\dots,d\}$, let $V_{E}=\{+e_i:i\in E\}\cup \{-e_j: j\in \{1,2,\dots, d\}\setminus E\}$ and let $B_{E}=-V_{E}=\mc{E}\setminus V_{E}$.  Note that for each $E$, $V_{E}$ is an orthogonal set of vectors. 
If no such $V$ in the statement of the lemma exists, then $\mu(\{A:A\subset B_{E}\})=\mu(\{A:A\cap V_{E}=\emptyset\})>0$ for each $E$.   Let $F=\{0,1\}^d\subset \Z^d$.  For $x\in F$, let $E(x)=\{i \in \{1,\dots,d\}: x^{[i]}=1\}$.  Then with positive $\nu$ probability, $\mc{G}_x\subset B_{E(x)}$ for every $x\in F$.  It is easy to check that on this event we have that $\mc{C}_o \subset F$.
\qed\medskip

\medskip
Thus according to Lemma \ref{lem:criteriontobeinfinite}, models satisfying $\theta_+=1$ contain subnetworks (determined by $V$) of random walks.  These walks will typically coalesce as in \cite{Arrat} and \cite{TW}.  Of most relevance to us is the 2-dimensional setting. 
\begin{LEM}
\label{lem:coalescence1}
In the model $(\uparrow\rightarrow)$, $\mc{C}_x\cap\mc{C}_y\neq\emptyset$ $\nu$-almost surely for every $x,y$. 
\end{LEM}

\proof Assume first that $x$ and $y$ both belong to the line $\{(i,j):i+j=0\}$. Follow the unique path from $x$ (resp. $y$), and after $n$ steps let $X_n$ (resp. $Y_n$) be the first coordinate of the point reached. Then $X_n$ and $Y_n$ follow independent random walks (up to the time they coalesce), with probability $p$ of standing in place, and probability $1-p$ of moving a step to the right. So $X_n-Y_n$ is a random walk, absorbed at 0, which moves $+1$ or $-1$ with probability $p(1-p)$ each, and otherwise stands in place. Since this nearest neighbour RW is symmetric, it hits 0 with probability 1, which is the desired conclusion. 

If $x^{[1]}+x^{[2]}\neq y^{[1]}+y^{[2]}$, just follow the path from one point till it reaches the diagonal line the other starts on, and then apply the same argument.\qed\medskip

An easy consequence of Lemmas \ref{lem:criteriontobeinfinite} and \ref{lem:coalescence1} is the following result, whose proof is omitted.

\begin{COR}
\label{cor:coalescence2}
Suppose that $d=2$, $\theta_+=1$ and $\mu$ is at least 2-valued.
Then $\mc{C}_x\cap\mc{C}_y\neq\emptyset$, $\nu$-almost surely, for every $x$ and $y$, except for the model $(\leftrightarrow\rightarrow)$ (and its rotations).
\end{COR}


\subsection{Percolation}
\label{sec:percolation}
\begin{figure}
\vspace{-.1cm}
\includegraphics[scale=.5]{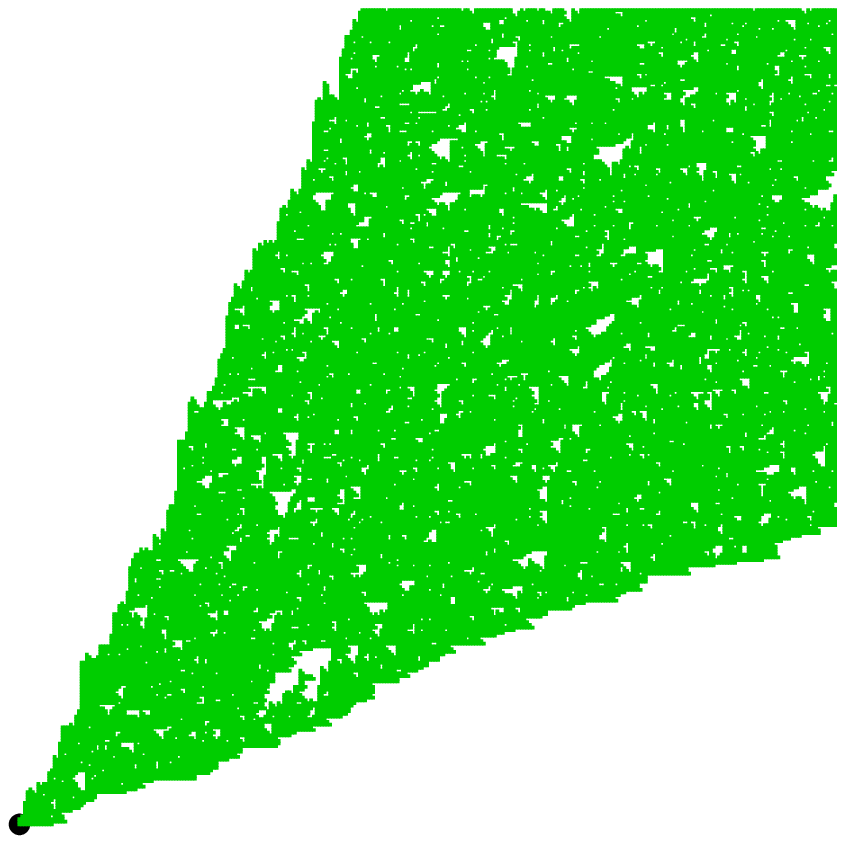}
\includegraphics[scale=.5]{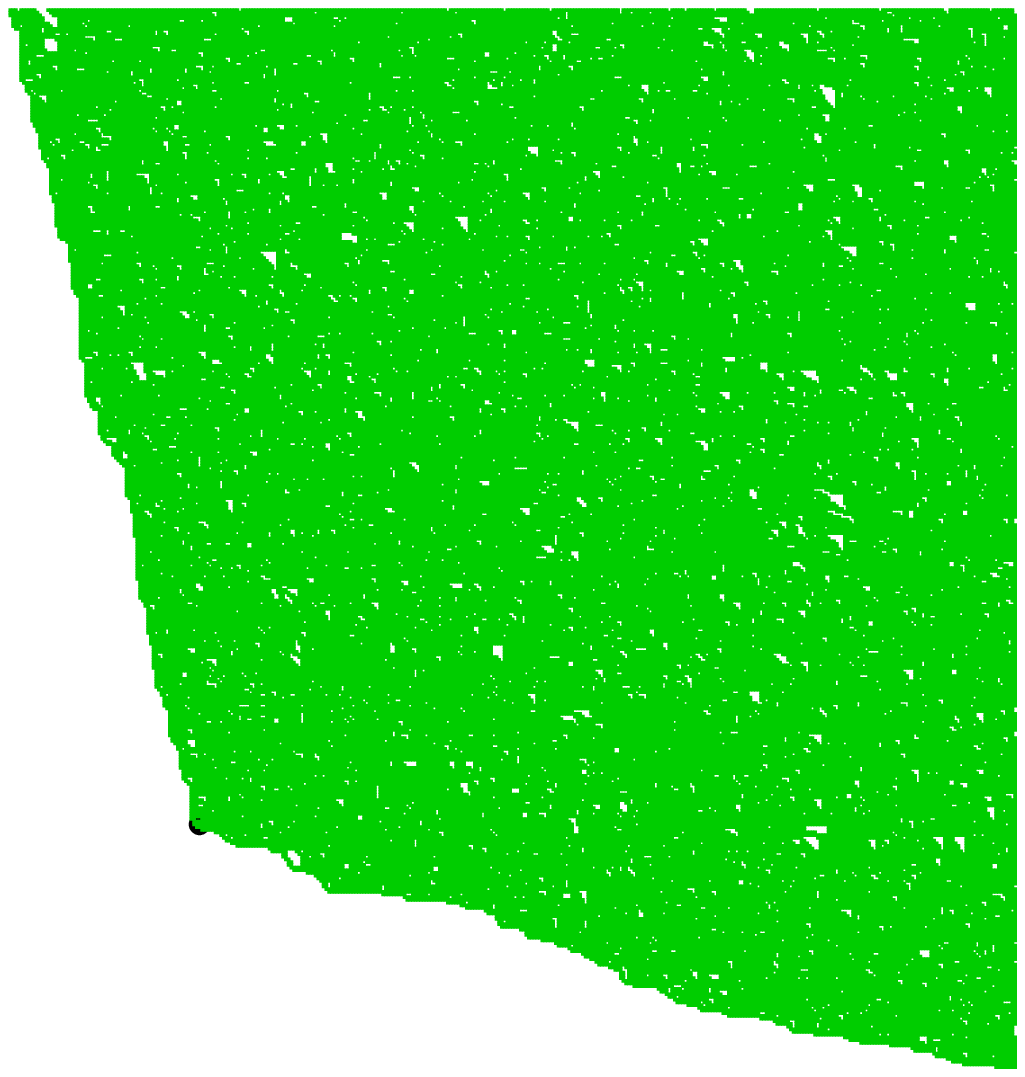}
\vspace{-3cm}
\caption{Part of a realisation of the set $\mc{C}_o$ for (a) the model with $\mu(\{\uparrow,\rightarrow\})=.8=1-\mu(\emptyset)$, and (b) the model with $\mu(\{\uparrow,\rightarrow\})=.8=1-\mu(\{\downarrow,\leftarrow\})$.  The former is a subset of the northeast quadrant, while (loosely speaking) the latter contains the northeast quadrant.}
\label{fig:ospC_o}
\end{figure}

In two dimensions the non-trivial site-percolation models (2-valued models with $A_2=\emptyset$) that fit into our framework are $(\NE\cdot)$, $(\SWE\cdot)$, and $(\NSEWalt\cdot)$.  Recall that the first is 
oriented site percolation (see Durrett \cite{Durrett}), while the third is  
site percolation (see Grimmett \cite{Grim99}).  The intermediate model $(\SWE\cdot)$, where $\mu(\{\leftarrow, \downarrow, \rightarrow\})=p$ and $\mu(\emptyset)=1-p$, is partially-oriented site perc.~(see Hughes \cite{Hughes} or M\'artin and Vannimenus \cite{MarV}).
All three models are {\sl monotone} in the sense that they can be coupled so that the set of arrows at each site is a non-decreasing function of $p$. Thus it is immediate that there is a critical $p_c$ in each case, such that $\theta_+>0$ if $p>p_c$, and $\theta_+=0$ a.s.~if $p<p_c$. We denote these critical probabilities by 
$p_c^{\smallNE}$, $p_c^{\smallSWE}$, and $p_c^{\smallNSEWalt}$. Clearly 
$p_c^{\smallNE}\ge p_c^{\smallSWE}\ge p_c^{\smallNSEWalt}$. 

Estimates are 
$p_c^{\smallNE}\approx 0.7055$, $p_c^{\smallSWE}\approx 0.6317$ and $p_c^{\smallNSEWalt}\approx 0.5927$. 
See Hughes \cite{Hughes} for references.
The best rigorous bounds the authors are aware of are that $p_c^{\smallNE}\in[0.6882,0.7491]$, $p_c^{\smallSWE}\in[0.5972,0.7491]$, and $p_c^{\smallNSEWalt}\in[0.5416,0.6795]$.
Gray, Smythe, and Wierman \cite{GSW} give the lower bounds for $p_c^{\smallNE}$ and $p_c^{\smallSWE}$. Balister, Bollob\'as, and Stacey \cite{BBS} give the upper bound for $p_c^{\smallNE}$, which implies that for $p_c^{\smallSWE}$. The $p_c^{\smallNSEWalt}$ bounds are from Men'shikov and Pelikh \cite{MenP} and Wierman \cite{Wierman95}. 
In Section \ref{sec:M_x} we will establish rigorous bounds on certain other critical values, that appear to improve bounds in the literature. 

For the model $(\NSEWalt\,\cdot)$, the cluster of $o$ in the usual site-percolation sense is our $\mc{M}_o$.   
Moreover, $\mc{B}_o=\mc{M}_o$ provided $\mc{G}_o=\NSEWalt$, and points in $\mc{C}_o$ are either in $\mc{M}_o$ or are neighbours of such points. These statements all follow because in this model, any connected path of $\NSEWalt$ sites is necessarily connected in both directions.  The following result is a kind of generalisation of this idea.  
\begin{LEM}
\label{lem:B_oC_o}
Suppose there exists $A\ne \emptyset$ such that $\mu(A)=p$ and $\mu(\emptyset)=1-p$.  Then $\theta_+=p\theta_-$. 
\end{LEM}
\proof 
Given an environment $\mc{G}$, define an environment $\ffG$ (having the same law as $\mc{G}$) by
$\ffG_x=\mc{G}_{-x}$.  
Let $\tilde{\mc{G}}$ 
be the environment obtained from $\mc{G}$ by replacing $\mc{G}_o$ by $\emptyset$.
Note that 
\begin{equation}
\lbeq{connections}
\{|\mc{C}_o|=\infty\}=\{\mc{G}_o=A\}\cap \big\{\exists x_1\in \{o+v:v\in A\}: |\mc{C}_{x_1}|=\infty \text{ in }\tilde{\mc{G}}\big\}.
\end{equation}
The two events on the right of \refeq{connections} are independent since the first depends only on $\mc{G}_o$, while the latter depends only on $\tilde{\mc{G}}$.


The second event on the right of \refeq{connections} occurs if and only if there is an infinite self-avoiding path $\vec{x}=\{x_i\}_{i\ge 1}$ of sites such that
for each $i\ge 1$, $\mc{G}_{x_{i}}=A$ and $x_{i}=x_{i-1}+v_i$ (with $x_0\equiv o$) for some $v_i\in A$. 
Such a path $\vec{x}$ exists in $\mc{G}$ if and only if 
the path $-(\vec{x})=\{-x_i\}_{i\ge 1}$
is such that for each $i\ge 1$, $\ffG_{-x_{i}}=A$ and $-x_{i-1}=-x_{i}+v_i$ (with $x_0\equiv o$) for some $v_i\in A$, i.e.~the path $-(\vec{x})$
is an infinite open path to the origin in $\ffG$.  
Thus we have shown that
\[\big\{\exists x_1\in \{o+v:v\in A\}: |\mc{C}_{x_1}|=\infty \text{ in }\tilde{\mc{G}}\big\}=\big\{|\mc{B}_o|=\infty \text{ in }\ffG\big\}.\]
It follows that
\eqalign
\nu(|\mc{C}_o|=\infty\})=&\nu(\{\mc{G}_o=A\})\nu\big(\exists x_1\in \{o+v:v\in A\}: |\mc{C}_{x_1}|=\infty \text{ in }\tilde{\mc{G}}\big)\nn\\
=&p\nu\big(|\mc{B}_o|=\infty \text{ in }\ffG\big)=p\nu\big(|\mc{B}_o|=\infty \text{ in }\mc{G}\big)=p\nu\big(|\mc{B}_o|=\infty\big)\nn.
\enalign
\qed

Note that this result also holds on the triangular lattice in 2-dimensions (this fact will be used in the next section).

\section{The set of points $\mc{B}_y$ from which $y$ can be reached}
\label{sec:B_x}
There are cases in which points can only ever be reached from finitely many locations. This is known in the case of coalescing random walks (see \cite{TW}). We require only the 2-dimensional version.
\begin{LEM}
\label{lem:finiteB}
In the model $(\uparrow\,\rightarrow)$ with $p\in (0,1)$, $\theta_-=0$.
\end{LEM}

\proof 
Let $L_n=\{x:x^{[1]}+x^{[2]}=-n\}$. Set $X_n=\#\{x\in L_n: x\to o\}$, and let $\mc{F}_n$ be the $\sigma$-field generated by the environment on or above $L_n$. Then 
\begin{align*}
\mE[X_{n+1}\mid \mc{F}_n]
&=\sum_{x\in L_{n+1}}\Big(p1_{\mc{B}_o}(x+e_2)+(1-p)1_{\mc{B}_o}(x+e_1)\Big)\\
&=\sum_{y\in L_n\cap\mc{B}_o}\Big(p+(1-p)\Big)
=X_n.
\end{align*}
Thus $X_n$ is a non-negative 
martingale with respect to $\mc{F}_n$, whence it converges as $n\to\infty$, and the only possible limit is 0.
\qed

As with Lemma \ref{lem:exist_inf_C}, since $\{\exists x:|\mc{B}_x|=\infty\}$ is a tail event, we have the following result.
\begin{LEM}
\label{lem:exist_inf_B}
If $\theta_->0$ then $\nu(\exists x:|\mc{B}_x|=\infty)=1$.
\end{LEM}

We now turn to a class of results, giving environments under which $\theta_->0$. We start with a trivial criterion which applies e.g.~to the 2-valued model $(\NE\uparrow)$.
\begin{align}
\text{If there is an }e \text{ such that }\mu(\{A:e\in A\})=1 \text{ then }\theta_-=1.
\label{BhasAhalfline}
\end{align}
More interesting are the cases where $\theta_-\in (0,1)$, e.g.~the model $(\WE\,\uparrow)$ (see Figure \ref{fig:WE_N_B_o}).  
\begin{figure}
\vspace{-3cm}
\includegraphics[scale=.5]{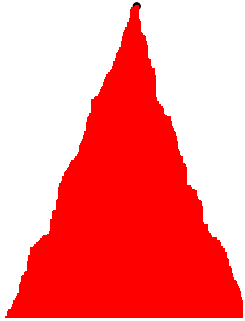}
\vspace{-.5cm}
\caption{Parts of realisations of the set $\mc{B}_o$ for the model $(\WE\,\uparrow)$ 
 for two values of $p$ ($p=.7$ and $p=.3$). Centred at $o$. NOTE: $p=.7$ version deleted because of lack of space. Consult the published copy or authors' websites for a complete version.}
\label{fig:WE_N_B_o}
\end{figure}
\begin{PRP}
\label{prp:infiniteB1}
For the model $(\WE\,\uparrow)$ with $p\in (0,1)$:
\begin{enumerate}
\item $\theta_-\in (0,1)$;
\item almost surely, on the event that $\mc{B}_o$ is infinite, there exists an infinite open path $\{x_{-n}\}_{n\ge 0}$ ending at $o$ with 
$x_{-n}^{[2]}\to-\infty$ (monotonically); 
\item $\nu$-a.s., if $|\mc{B}_x|=|\mc{B}_y|=\infty$ then $|\mc{B}_x\cap\mc{B}_y|=\infty$. 
\end{enumerate}
\end{PRP}
\proof
Let $\mc{L}_n=\Z\times \{-n\}$ and define $L_n$ and $U_n$ to be the infimum and supremum of the projection of $\mc{L}_n\cap \mc{B}_o$ on the 1st coordinate axis. 
Of course, if this set is empty then $L_n=\infty$ and $U_n=-\infty$.    
We claim that for each $n\ge 0$,
\begin{align}
\mc{L}_n\cap \mc{B}_o=\big([L_n,U_n]\cap\Z\big)\times\{-n\}.\label{LUblock}
\end{align}
The claim is established by induction, with the case $n=0$ being trivially true since there are no downward arrows.  

Assume this statement for $n$.  Suppose there is at least one $z\in[L_n,U_n]$ such that $\mc{G}_{(z,-(n+1))}=\{\uparrow\}$. Then $(z,-(n+1))$ connects to $o$ as well, as does any $(w,-(n+1))$ which connects to $(z,-(n+1))$ by a sequence of $\rightarrow$ or $\leftarrow$. Thus $(w,-(n+1))$ connects to $o$ whenever $L_n\le w\le U_n$, either directly or via such a $z$.  This is also the case for any $w=(w_1,-(n+1))$ such that for every $k\in [w_1,L_n)\cap \Z$, the environment at $(k,-(n+1))$ is $\WE$, but not other vertices to the left of $(L_n,-(n+1))$.  Similarly for $w$'s to the right of $U_n$ at level $-(n+1)$. In other words, $L_{n+1}\le L_n$ and $U_n\le U_{n+1}$, and the set of $(w,-(n+1))$ connecting to $o$ forms an interval. On the other hand, if there is no $z\in[L_n,U_n]$ such that $\mc{G}_{(z,-(n+1))}=\{\uparrow\}$, then no vertex with 2nd coordinate $-(n+1)$ connects to $o$ at all.  Therefore for each $n$, either this interval expands ($L_{n+1}\le L_n$ and $U_n\le U_{n+1}$) or it disappears altogether ($=\emptyset$).  This verifies (\ref{LUblock}).

Consider the number of integers $D_n=(U_n-L_n+1)_+$ in the interval $[L_n,U_n]$.  It is easily checked that $D_n$ is a Markov chain, that transitions $k\mapsto 0$ have probability $\alpha(k)=p^k$, that transitions $k\mapsto k$ have probability $(1-p)^2(1-p^k)$, and therefore that all other transitions combined have probability 
$$
\beta(k)=1-p^k-(1-p)^2(1-p^k)=(1-p^k)p(2-p)\ge (1-p)p(2-p)=c>0.
$$
Set $T_0=0$, and let $T_{k+1}=\min\{n>T_k:D_n\neq D_{T_k}\}$ be the times $D_n$ changes values. Clearly $D_0\ge 1$, and by induction, $D_{T_k}$ is either at least $k$ or it equals $0$. Therefore
$$
\nu(D_{T_{k+1}}>D_{T_k}\mid\mc{F}_{T_k})
=\frac{\beta(D_{T_k})}{\alpha(D_{T_k})+\beta(D_{T_k})}=1-\frac{\alpha(D_{T_k})}{\alpha(D_{T_k})+\beta(D_{T_k})}\ge 1-\frac{p^k}{c}.
$$
Choose $\kappa$ so large that $p^{\kappa}<c/2$, and $\gamma$ such that $e^{-\gamma t}<1-t$ for $0\le t\le 1/2$. Then the above expression is $\ge e^{-p^k\gamma /c}$ for $k\ge \kappa$. So by the strong Markov property, and convergence of $\sum p^k$, 
$$
\nu(D_n>0\quad\forall n)\ge \nu(D_{T_\kappa}>0)\prod_{j=\kappa}^\infty e^{-p^k\gamma/c}>0.
$$
Thus in fact $D_n>0$ for every $n$, with positive probability. Whenever all $D_n>0$, it follows that $D_n\to\infty$ and $\mc{B}_o$ is infinite. Thus $\theta_->0$.
To see that it is $\theta_+<1$ as well, just observe that the configuration $\mc{G}_{-e_1}=\uparrow=\mc{G}_{e_1}$, $\mc{G}_{-e_2}=\leftrightarrow$ establishes that $\nu(\mc{B}_o=\{o\})>0$. 

For future reference, notice that it follows from our proof that when $\mc{B}_o$ is infinite, it is almost surely also the case that $L_n\downarrow-\infty$ and $U_n\uparrow\infty$. 

To obtain a semi-infinite path through $\mc{B}_o$, observe that we have at least one finite path from $(U_n,-n)$ to $o$,
 for each $n$.  These can in fact be chosen to form a monotone sequence of paths, in the sense that if two such paths ever meet, we make them coalesce. It follows that the paths so chosen converge as $n\to\infty$. The limit is the desired semi-infinite path.

Finally, the fact that $\mc{B}_x\cap \mc{B}_y$ is infinite, whenever $\mc{B}_x$ and $\mc{B}_y$ are follows immediately from the monotonicity of $L_n$ and $U_n$, and the fact that $D_n\to\infty$. 
\qed


We now will establish the same type of result, for the model $(\NE\,\leftarrow)$.  See Figure \ref{fig:NE_W_B_o_proof}. 
Between them, Propositions \ref{prp:infiniteB1} and \ref{prp:infiniteB2} will allow us to decide whether $\theta_-\in (0,1)$, for many 2-valued 2-dimensional models. See Table \ref{tab:connections} for more details. 

\begin{PRP}
\label{prp:infiniteB2}
The assertions of Proposition \ref{prp:infiniteB1} also hold for the model $(\NE\,\leftarrow)$ with $p\in (0,1)$.
\end{PRP}
\proof
For $n\ge 0$, let $\mc{L}_n$, $L_n$ and $U_n$ be defined as in the proof of Proposition \ref{prp:infiniteB1}.
We will again show by induction that (\ref{LUblock}) holds.
\begin{figure}
\begin{center}
\includegraphics[scale=.45]{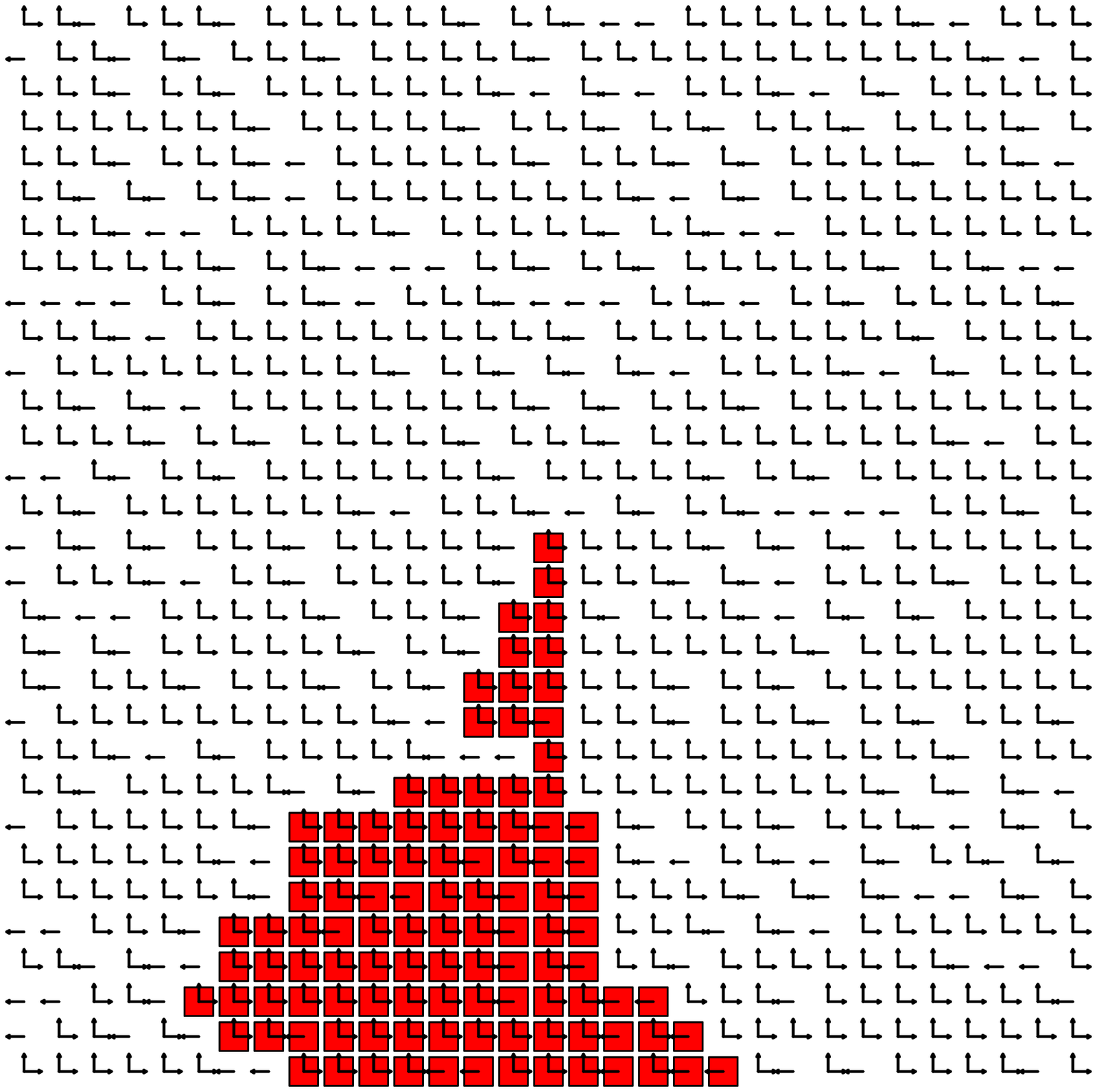}
\hspace{1cm}
\includegraphics[scale=.45]{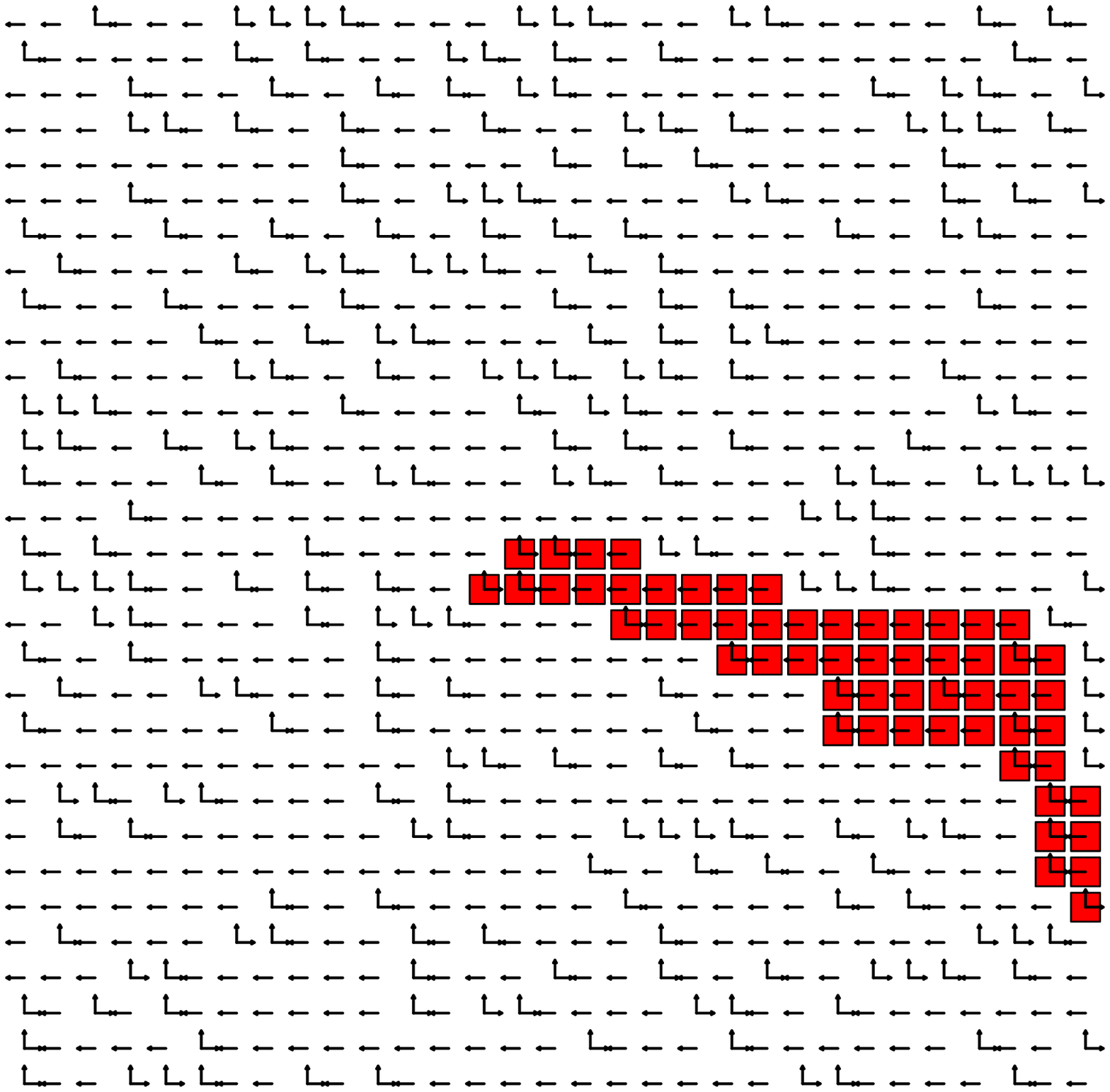}
\end{center}
\vspace{-.5cm}
\caption{Parts of realisations of the set $\mc{B}_o$ for the model with $\mu(\{\uparrow,\rightarrow\})=p=1-\mu(\{\leftarrow\})$ for two values of $p$ ($p=.7$ and $p=.3$). Centred at $o$.}
\label{fig:NE_W_B_o_proof}
\end{figure}

So assume (\ref{LUblock}) for $n$, and consider $\mc{B}_o\cap \mc{L}_{n+1}$. We will examine several cases separately.  First, suppose $l\le L_n\le U_n\le u$. Then $L_{n+1}=l$ and $U_{n+1}=u$ if and only if the following conditions hold: 
\begin{align*}
\mc{G}_{(l-1,-(n+1))}&=\leftarrow \\
\mc{G}_{(j,-(n+1))}&=\NE\quad\text{ for $l\le j\le L_n$}\\
\mc{G}_{(j,-(n+1))}&=\leftarrow\quad\text{ for $U_n<j\le u$}\\
\mc{G}_{(u+1,-(n+1))}&=\NE.
\end{align*}
To see this, observe that $x=(j,-(n+1))\notin\mc{B}_o$ for $j<l$, since tracing a path from $x$ can only reach $\mc{L}_n$ strictly to the left of $L_n$, and hence outside $\mc{B}_o$. Likewise 
$x=(j,-(n+1))\notin\mc{B}_o$ for $j>u$. For $l\le j\le L_n$ we can step to the right till we reach $(L_n,-(n+1))$ and then step up to $(L_n,-n)$. Thus $(j,-(n+1))\in\mc{B}_o$. For $L_n\le j\le u$ consider a path that steps either left or up. The first step up will be at a point $(j',-(n+1))$ with $L_n\le j'\le U_n$, so by induction the point $(j',-n)$ we reach will lie in $\mc{B}_o$. Thus $\mc{B}_o\cap \mc{L}_{n+1}$ forms a contiguous block with this scenario. 

Now suppose that $L_n<l\le U_n\le u$. Then the same argument shows that $L_{n+1}=l$ and $U_{n+1}=u$, and $(j,-(n+1))\in\mc{B}_o$ for $l\le j\le u$, provided the following conditions hold: 
\begin{align*}
\mc{G}_{(j,-(n+1))}&=\leftarrow\quad\text{ for $L_n\le j<l$}\\
\mc{G}_{(l,-(n+1))}&=\NE\\
\mc{G}_{(j,-(n+1))}&=\leftarrow\quad\text{ for $U_n<j\le u$}\\
\mc{G}_{(u+1,-(n+1))}&=\NE.
\end{align*}
Between them, the above scenarios cover all cases in which there is a $L_n\le j\le U_n$ with 
$\mc{G}_{(j,-(n+1))}=\NE$. So the only remaining possibility is that 
$$
\mc{G}_{(j,-(n+1))}=\leftarrow\quad\text{ for $L_n\le j\le U_n$,}
$$
in which case $\mc{B}_o\cap \mc{L}_{n+1}=\emptyset$, so $L_{n+1}=+\infty$, $U_{n+1}=-\infty$. This verifies (\ref{LUblock}).

Let $l'\le u'$. We can now also read off $\nu(L_{n+1}=l, U_{n+1}=u\mid L_n=l', U_n=u')$, getting 
\begin{align*}
(1-p)p^{l'-l+1}(1-p)^{u-u'}p=p^{l'-l+2}(1-p)^{u-u'+1}, &\quad\text{if $l\le l'\le u'\le u$, }\\
(1-p)^{l-l'}p(1-p)^{u-u'}p=p^{2}(1-p)^{l+u-l'-u'}, &\quad\text{if $l'< l\le u'\le u$, }\\
(1-p)^{u'-l'+1}&\quad\text{if $l=+\infty$, $u=-\infty$.}
\end{align*}

We now couple these to a pair of independent random walks $L_n^0$ and $U_n^0$ that evolve as follows: If $U_n^0=u'$ then $U_{n+1}^0=u'+j$ with probability $p(1-p)^j$, for $j\ge 0$. If $L_n^0=l'$ and $j\ge 1$ then $L_{n+1}^0=l'+j$ with probability $p(1-p)^j$. While if $L_n^0=l'$ and $j\ge 0$ then $L_{n+1}^0=l'-j$ with probability $p^{j+1}(1-p)$. It follows that for $l'\le u'$ we have
$$
\nu(L_{n+1}=l, U_{n+1}=u\mid L_n=l', U_n=u')
=\nu(L_{n+1}^0=l, U_{n+1}^0=u\mid L_n^0=l', U_n^0=u'),
$$
provided $l\le u'\le u$. Put another way, let $T$ be the first $n$ (if any) such that $\mc{L}_n\cap \mc{B}_o=\emptyset$, and $T^0$ be the first $n$ such that $L_n^0>U_{n-1}^0$. Then the process 
$(L_n,U_n)_{0\le n<T}$ has the same law as the process  $(L_n^0,U_n^0)_{0\le n<T^0}$.

Consider statement (a) of the proposition. $\mc{B}_o$ is infinite exactly when $T=\infty$, so we want to show that $\nu(T^0=\infty)>0$. But $T^0$ is the first time the random walk $D_n^0=L_n^0-U_{n-1}^0$ hits $[1,\infty)$. So this result boils down to showing that the random walk $D_n^0$ drifts to the left, or in other words, that 
$$
\mE[\Delta L_n^0]<\mE[\Delta U_n^0],
$$
where e.g.~$\Delta L_n^0= L_n^0- L_{n-1}^0$ for $n\in \N$.  In fact, 
$$
\mE[\Delta L_n^0]=\sum_{j\ge 1}j(1-p)^jp-\sum_{j\ge 0}jp^{j+1}(1-p)
=\frac{1-p}{p}-\frac{p^2}{1-p}
$$
and
$$
\mE[\Delta U_n^0]=\sum_{j\ge 0}j(1-p)^jp
=\frac{1-p}{p}.
$$
So by the above reasoning, $\theta_-\in (0,1)$. 

Statement (b) follows as in the proof of Proposition \ref{prp:infiniteB1}. To obtain (c) we use the law of large numbers, and the comparison with $L_n^0$ and $U_n^0$. This shows that whenever $\mc{B}_o$ is infinite, in fact $U_n/n\to \mE[\Delta U_n^0]=(1-p)/p$, and likewise $L_n/n\to (1-p)/p-p^2/(1-p)$. So suppose $\mc{B}_x$ and $\mc{B}_y$ are both infinite. Without loss of generality, we'll assume that $x^{[2]}=y^{[2]}$ and $x^{[1]}<y^{[1]}$. Let $L_n(y)$ (resp.~$U_n(x)$) be the lower (resp.~upper) process obtained from our construction, starting not from $o$ but from $x$ (resp. $y$). Since the asymptotic speed of $L_n(y)$ is less than the asymptotic speed of $U_n(x)$, eventually $U_n(x)>L_n(y)$, providing infinitely many common elements to $\mc{B}_x$ and $\mc{B}_y$. 
\qed
\medskip

\begin{COR}
\label{cor:infiniteB2}
$\theta_-\in (0,1)$ 
for the models $(\NE \SW)$, $(\WE \, \NS)$, and $(\SWE \uparrow)$ whenever $p \in (0,1)$.  
\end{COR}

\proof
Models $(\WE\, \NS)$ and $(\SWE\,\uparrow)$ contain the model $(\WE\,\uparrow)$, while $(\NE \SW)$ contains the model $(\NE\,\leftarrow)$.  Thus $\theta_->0$ by Propositions \ref{prp:infiniteB1} and \ref{prp:infiniteB2}. In each case it is easy to find trapping configurations showing $\nu(\mc{B}_o=\{o\})>0$.\qed\medskip

In order to describe the possible structures of $\mc{B}_x$ clusters, we make the following definition.
\begin{DEF}
\label{def:blocked}
Given $w:\Z\ra \Z$, define $w_{\le}\subset \Z^2$ and $w_{>}\subset \Z^2$ by
$$
w_\le = \{y\in \Z^2: y^{[2]}\le w(y^{[1]})\} \quad \text{ and }\quad w_> = \{y\in \Z^2: y^{[2]}> w(y^{[1]})\}.
$$
We say that $y$ is {\sl below $w$} if $y^{[2]}\le w(y^{[1]})$, and {\sl strictly below $w$} if $y^{[2]}< w(y^{[1]})$. 
We say that $w$ is an {\sl upper blocking function (ubf)} for $\mc{G}$ if there is no open path in $\mc{G}$ from $w_>$ to $w_\le$.
If $x$ is below an upper blocking function, we say that $\mc{B}_x$ is {\sl blocked above}.
\end{DEF}
Defining $w_{\ge}$ and $w_<$ similarly, $w$ is a {\sl lower blocking function (lbf)} for $\mc{G}$ if there is no  path in $\mc{G}$ from $w_<$ to $w_\ge$, and if $x$ is above a lbf, then $\mc{B}_x$ is {\sl blocked below}.
The reason for the terminology is the following trivial result.
\begin{LEM}
\label{lem:blocked}
If $w$ is an upper blocking function, and $x$ is below $w$, then $\mc{B}_x\subset w_{\le}$. Likewise, if $w$ is a lower blocking function and $x$ is above $w$, then $\mc{B}_x\subset w_{\ge}$.
\end{LEM} 

For $C\subset \mc{E}$, let 
$\mc{A}_C=\{A\subset\mc{E}: A\cap C\ne \emptyset\}$.  
We use the shorthand notation $\mc{A}_{\smallNE}$ for $\mc{A}_{\{e_1, e_2\}}$, and $\mc{A}_{e}$ for $\mc{A}_{\{e\}}$.  We now reveal the possible forms of $\mc{B}_x$ for certain models (see e.g.~Figure \ref{fig:orthant_B_o}).
\begin{figure}
\begin{center}
\includegraphics[scale=.45]{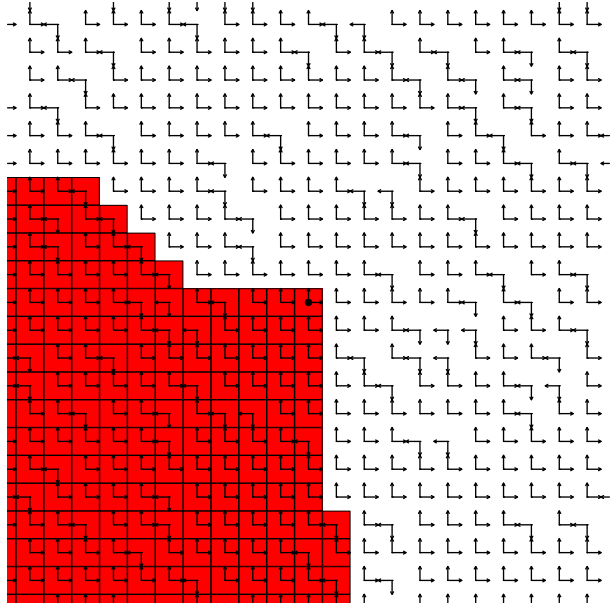}
\end{center}
\caption{Part of a realisation of the set $\mc{B}_o$ for the model with $\mu(\{\uparrow,\rightarrow\})=.8=1-\mu(\{\leftarrow,\downarrow\})$. Centred at $o$.}
\label{fig:orthant_B_o}
\end{figure}
\begin{PRP}
\label{PRP:trichotomy}
Fix $d=2$.  Suppose that $\mu(\mc{A}_{e})>0$ for each $e\in \mc{E}$, and $\mu(\mc{A}_{\smallNW})=\mu(\mc{A}_{\smallSE})=1$. 
Then 
\begin{enumerate}
\item 
$\nu$-a.s.~one of the following occurs:
  \begin{enumerate}
  \item[(i)] $\mc{B}_x$ is finite;
  \item[(ii)] $\mc{B}_x=\Z^2$;
  \item[(iii)] there exists a decreasing ubf $\,W:\Z\to\Z$ such that $\mc{B}_x=W_{\le}$;
  \item[(iv)] there exists a decreasing lbf $\,W:\Z\to\Z$ such that $\mc{B}_x=W_{\ge}$.
  \end{enumerate}
\item At most one of (ii), (iii), (iv) can have positive probability. 
\item $\nu$-a.s.~if $|\mc{B}_x|=|\mc{B}_y|=\infty$ then $|\mc{B}_x\cap\mc{B}_y|=\infty$. 
\end{enumerate}
\end{PRP}
\begin{proof} 
Without loss of generality, $x=o$. 
Suppose $y,z\notin\mc{B}_o$, with $y^{[1]}=z^{[1]}$ and $y^{[2]}<z^{[2]}$. Since $\mu(\mc{A}_{\smallSE})=1$, we may find SE paths from both $y$ and $z$ that are consistent with the environment, but can be chosen to arise from a model $\tilde\mu$ depicted by $(\downarrow \, \rightarrow)$.  
This can be achieved by choosing $\downarrow$ and $\rightarrow$ independently at random (using the same probability $q$) at any vertex where both occur.
We may now apply (a rotation of) Lemma \ref{lem:coalescence1} to see that the SE path from $y$ crosses the SE path from $z$ with probability 1. Since these paths lie within $\mc{C}_y$ and $\mc{C}_z$ respectively, and both $y,z\in\mc{B}_o^c$, it follows that both paths lie entirely in $\mc{B}_o^c$ as well. Following one from $y$ to the intersection point, and then the other backwards in time to $z$ produces a simple polygonal path from $y$ to $z$, all of whose vertices belong to $\mc{B}_o^c$.  Similarly we may also find intersecting NW paths from $y$ and $z$ that use only the moves $\uparrow$ and $\leftarrow$. Following one path from $z$ to the intersection point, and then the other back to $y$ produces a simple polygonal path which also lies entirely within $\mc{B}_o^c$. Concatenating the two paths gives us a cycle in $\mc{B}_o^c$ whose vertices lead from $y$ to $z$ and then back to $y$. 

Now suppose that $w\in\mc{B}_o$, with $w^{[1]}=y^{[1]}=z^{[1]}$ but $y^{[2]}<w^{[2]}<z^{[2]}$.  There is by definition an open path from $w$ to $o$ which lies entirely in $\mc{B}_o$. This path cannot cross the above cycle, from which we conclude that $o$ is itself enclosed by the cycle. It follows that $\mc{B}_o$ is also enclosed by the cycle, and hence that $\mc{B}_o$ is finite. That is, in this scenario, condition (i) holds. 

To put this a different way, suppose that $\mc{B}_o$ is infinite.
The argument above establishes that for every $n\in\Z$, the set of points $y\in\mc{B}_o^c$ such that $y^{[1]}=n$ forms a vertical interval $\{n\}\times (L_n,U_n)$. Case (ii) above corresponds to this interval being empty for every $n$. So suppose further that the interval is non-empty for some $n$. Then constructing SE and NW paths from that point (as above) shows that the interval is in fact non-empty for every $n$. Even better, running the SE path backwards in time and the NW path forwards in time gives a simple polygonal path within $\mc{B}_o^c$ that crosses every vertical line in $\Z^2$. 

If $o$ lies below this path, then we must have $U_n=+\infty$ for every $n$, as any path to $o$ from above our path would have to cross the latter. We will show that case  (iii) holds with $W(n)=L_n$.  With this choice of $W$ we first show that $W(n)>-\infty$ for each $n$.
Since $\mc{B}_o$ is assumed to be infinite, we have $W(n)>-\infty$ for some $n$. Suppose $W(m)=-\infty$ for some $m>n$. Choosing the smallest such $m$, we have $W(m-1)>-\infty$, and that $(m-1,k)\in\mc{B}_o$ for every $k\le W(m-1)$, but that $(m,k)\notin\mc{B}_o$ for any $k$. In particular, there is no $\leftarrow$ in any $\mc{G}_{(m,k)}$, $k\le W(m-1)$, since if there were, following that move from $(m,k)$ would lead into $\mc{B}_o$, from which we could then reach $o$. This contradicts the assumption that $\mu(\mc{A}_{-e_1})>0$, since the latter easily implies that 
$$
\nu(\text{$\exists j_0,k_0$ such that $\mc{G}_{(j_0,k)}\notin\mc{A}_{-e_1}$ for every $k\le k_0$})=0.
$$
A similar argument, using that $\mu(\mc{A}_{e_1})>0$, rules out $W(m)=-\infty$ for $m<n$. It follows that $W(m)>-\infty$ for every $m$. In other words, $W:\Z\to\Z$. Therefore $\mc{B}_o=W_{\le}$, and since no open path can run from $\mc{B}_o^c$ to $\mc{B}_o$, it follows that W is a ubf.

To see that $W$ is decreasing, consider  $\mc{G}_{(n,W(n)+1)}$. Since $W(n)+1>L_n$, we must have $\downarrow\notin \mc{G}_{(n,W(n)+1)}$. Since $\mu(\mc{A}_{\smallSE})=1$, it follows  that 
$\rightarrow\in \mc{G}_{(n,W(n)+1)}$. Therefore $(n+1,W(n)+1)\notin \mc{B}_o$, so $W(n+1)\le W(n)$, which establishes (iii).


If $o$ lies above the constructed path, then the same argument shows that $-\infty=L_n<U_n<+\infty$ for each $n$, with $U_n$ decreasing, which puts us in case  (iv) with $W(n)=U_n$. This establishes  (a).


To prove  (b), suppose that 
$\nu(\text{$\mc{B}_o$ is blocked above})\ge\delta>0$. So 
$$
\nu(\text{$\exists$ an upper blocking function above $o$})\ge \delta.
$$ 
Choose $n\ge 1$. We may find a $k\ge 0$ such that 
$$
\nu(\text{$\exists$ an upper blocking function $w$ above $o$, such that $w(j)\ge -k$ for all $|j|\le n$})\ge\delta/2.
$$
By translation invariance of $\nu$, it follows that 
$$
\nu(\text{$\exists$ an upper blocking function above $[-n,n]^2$})
\ge\delta/2
$$
(just translate $\mc{G}$ upward by $k+n$). 
These are decreasing events, so in fact
$$
\nu(\text{$\forall n\ge 1, \exists$ an upper blocking function above $[-n,n]^2$})\ge\delta/2.
$$
But the latter is a tail event, 
so by the zero-one law, the probability is actually equal to 1.
We conclude that $\nu(\text{$\mc{B}_y$ is finite or blocked above})=1$ for every $y$.  Likewise, if 
$\nu(\text{$\mc{B}_o$ is blocked below})>0$, it follows that $\nu(\text{$\mc{B}_y$ is finite or blocked below})=1$  for every $y$.

If $\nu(\mc{B}_o=\Z^2)>0$ then $\nu(\text{$\exists$ upper blocking function above $o$})<1$. By translation invariance and what we have just shown, it follows that $\nu(\text{$\exists$ upper blocking function above $y$})=0$ for every $y$. Thus $\nu(\text{$\mc{B}_y$ is blocked above})=0$. Likewise $\nu(\text{$\mc{B}_y$ is blocked below})=0$. Therefore, $\nu(\text{$\mc{B}_y$ is finite or $=\Z^2$})=1$ for every $y$. 

Finally if $|\mc{B}_x|=\infty=|\mc{B}_y|$ then by (a) and (b) one of (ii)-(iv) holds for both $x$ and $y$ and (c) follows in each case.
\end{proof}
\begin{COR}
\label{COR:B_o_NS_EW}
For the model $(\updownarrow \,\leftrightarrow)$ with $p\in (0,1)$, we have that $\nu\big(\mc{B}_o=\Z^2\text{ }\big| \text{ }|\mc{B}_o|=\infty\big)=1$.
\end{COR}
\proof By Corollary \ref{cor:infiniteB2} the conditional probability is well defined.  By symmetry, the events  (iii) and (iv) in Proposition \ref{PRP:trichotomy}(a) have equal probability, which by (b) must equal 0.  
\medskip

In the following, note the differences with Proposition \ref{PRP:trichotomy}; we have fewer possible cases, but on the other hand the path $W(n)$ need not be decreasing (see e.g.~Figure \ref{fig:WSE_N_B_o}).
\begin{figure}
\begin{center}
\vspace{-3cm}
\includegraphics[scale=.7]{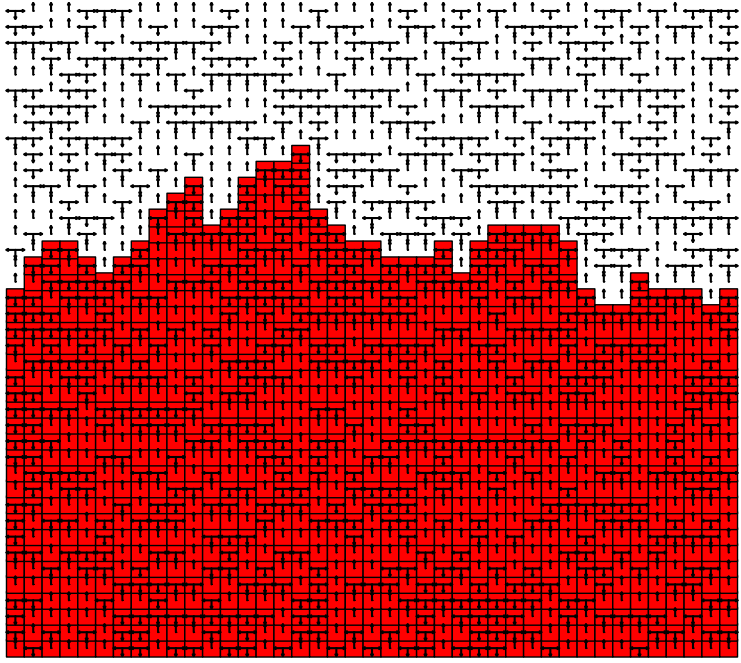}
\end{center}
\vspace{-2cm}
\caption{Part of a realisation of the set $\mc{B}_o$ for the model with $\mu(\{\downarrow,\leftarrow,\rightarrow\})=.5=\mu(\uparrow)$. Centred at $o$.}
\label{fig:WSE_N_B_o}
\end{figure}

\begin{COR}
\label{COR:trichotomy}
Fix $d=2$.  Suppose that
$\mu(\mc{A}_{\smallNE})=\mu(\mc{A}_{\smallNW})=1$, and $\mu(\mc{A}_{e})>0$ for $e\ne -e_2$. 
Then 
\begin{enumerate}
\item 
$\nu$-a.s.~one of the following occurs:\smallskip
  \begin{enumerate}
  \item[(i)] $\mc{B}_x$ is finite;
  \item[(ii)] $\mc{B}_x=\Z^2$;
  \item[(iii)] There exists a ubf $W:\Z\to\Z$ such that $\mc{B}_x=W_{\le}$.
  \end{enumerate}
\item At most one of (ii), (iii) can have positive probability. 
\item $\nu$-a.s., if $|\mc{B}_x|=|\mc{B}_y|=\infty$ then $|\mc{B}_x\cap\mc{B}_y|=\infty$. 
\end{enumerate}
\end{COR}
\proof
Simply replicate the proof of Proposition \ref{PRP:trichotomy}, using NE paths in place of SE paths. Everything goes through without change, except the property that $W(n)$ is decreasing. 

To see that the analogue of case (iv) of Proposition \ref{PRP:trichotomy} will not occur, suppose $\mc{B}_x$ is infinite and blocked below. Choose $y\in\mc{B}_x^c$ with $y^{[1]}=x^{[1]}$ and $y^{[2]}<x^{[2]}$. Follow the NE path from $y$ till it reaches some point $z$ with $z^{[2]}>x^{[2]}$. Then follow the NW path from $z$ till it reaches some point $\tilde z$ with $\tilde z^{[1]}=x^{[1]}$. By construction, $\tilde z^{[2]}>x^{[2]}$ so by  (iv) we must have $\tilde z\in\mc{B}_x$. But $y\to z\to\tilde z\to x$ implies that $y\to x$, which contradicts the fact that $y\notin\mc{B}_x$. Thus (iv) is impossible in this setting.
\qed

Additional assumptions may allow us to further restrict the possibilities. 
A trivial result of this type is:
\begin{COR}
\label{COR:trichotomy3}
In addition to the hypotheses of Corollary \ref{COR:trichotomy}, assume that $\mu(\mc{A}_{-e_2})=0$. Then for each $x$, $\nu$-a.s. either $\mc{B}_x$ is finite or it is blocked above. 
\end{COR}

In the remainder of this section, we explore some consequences of the above results for several models of particular interest.

Recall that in the triangular lattice, each vertex 6 neighbours.  
To construct oriented triangular site percolation (a model which we denote $\OTSP$), we declare each vertex in $\Z^2$ to be open with probability $p$, independently of each other vertex.
Closed vertices connect to no neighbours. Open vertices $x$ connect to 3 neighbours, $x-e_1$, $x+e_2$, and $x-e_1+e_2$.
There is a critical value $p_c^{\smallOTSP}$ such that oriented percolation clusters are finite when $p<p_c^{\smallOTSP}$ and are infinite with positive probability when $p>p_c^{\smallOTSP}$. 
The estimated value is $p_c^{\smallOTSP}\approx 0.5956$ (see De Bell and Essam \cite{DeBE} or Jensen and Guttmann \cite{JG}). 
We have not found rigorous bounds in the literature, though the following inequalities
\begin{equation}
\label{OTSPbounds}
0.5\le p_c^{\smallOTSP}\le 0.7491
\end{equation}
can be inferred from bounds on other models. To be precise, 
if we decrease the allowed bonds we get $p_c^{\smallOTSP}\le p_c^{\smallNE}\le 0.7491$ (the latter due to Balister et al \cite{BBS}). 
Similarly, if $p_c^{\text{TSP}}$ denotes the critical threshold for (un-oriented) triangular site percolation, then 
$p_c^{\smallOTSP}\ge p_c^{\text{TSP}}=1/2$ (see Hughes \cite{Hughes}). 
We will improve on the lower bound in Theorem \ref{prp:infiniteM3}, where we show that $p_c^{\smallOTSP}\ge0.5466$.  
\begin{THM}
\label{thm:infiniteBorthant}
For the model $(\NE \SW)$ with $p\in (0,1)$, we have $\theta_-\in (0,1)$, and $\nu$-a.s. 
\begin{enumerate}
\item If $1-p_c^{\smallOTSP}\le p\le p_c^{\smallOTSP}$ and $\mc{B}_o$ is infinite, then $\mc{B}_o=\Z^2$.
\item If $p>p_c^{\smallOTSP}$ and $\mc{B}_o$ is infinite, then there exists a decreasing ubf $W$ such that $\mc{B}_o=W_{\le }$.
\item If $p<1-p_c^{\smallOTSP}$ and $\mc{B}_o$ is infinite, then there exists a decreasing lbf $W$ such that $\mc{B}_o=W_{\ge }$.
\end{enumerate}
\end{THM}

\begin{proof} 
That $\theta_-\in (0,1)$ is contained in Corollary \ref{cor:infiniteB2}. Observe further that the hypotheses of Proposition \ref{PRP:trichotomy} hold. So (a) [resp.~(b), resp.~(c)] simply states which of (i) or (ii) [resp.~(i) or (iii), resp.~(i) or (iv)] of Proposition \ref{PRP:trichotomy} holds, for the given range of $p$.  In particular to obtain (b) [resp.~(c)] it is sufficient to show that $\mc{B}_o$ is blocked above [resp.~below].

Now let $w(n)$ be a decreasing function, and consider under what circumstances it can be an upper blocking function. Vertices in $w_\le^c$ which border $w_\le$ lie above or to the right of $w_\le$, and can be enumerated naturally to form a sequence of vertices moving upwards and to the left. More precisely, the possible transitions in this sequence of vertices are as follows. 
\begin{itemize}
\item Upwards, e.g.~from $(n,k)$ to $(n,k+1)$. This happens if $w(n)<k<w(n-1)$. 
\item Leftwards, e.g.~from $(n,k)$ to $(n-1,k)$. This happens if $w(n)=w(n-1)=k-1$. 
\item Diagonally to the NW, e.g.~from $(n,k)$ to $(n-1,k+1)$. This happens if $w(n)<k=w(n-1)$. 
\end{itemize}
We recognize these as the three connections from the vertex $(n,k)$ (if that vertex is open) in {\sl oriented triangular site percolation} model $\OTSP$.
For $w(n)$ to be an upper blocking function, it is necessary and sufficient that each vertex in this sequence have local environment $\NE$. Calling $\NE$ vertices ``open'' and $\SW$ vertices ``closed'', we have established the kind of duality relation that is familiar from percolation: upper blocking functions for our random environment correspond precisely to clusters for {\sl oriented triangular site percolation}. 

Thus upper blocking functions exist if and only if there are doubly infinite oriented percolation clusters. In other words, if and only if there are points $x$ such that $\mc{B}_x^{\smallOTSP}$ and $\mc{C}_x^{\smallOTSP}$ are both infinite (where the superscript indicates that we are referring to connections in the $\OTSP$ model described above).  Trivially this does not occur when $p<p_c^{\smallOTSP}$.  
Observe that the argument of Lemma \ref{lem:B_oC_o} applies equally well to the lattice $\OTSP$. If $p=p_c^{\smallOTSP}$ then 
$\mc{B}_x^{\smallOTSP}$ and $\mc{C}_x^{\smallOTSP}$ are both finite. This is the content of Theorem 1 of Grimmett and Hiemer \cite{GH02} (proved for
a different lattice, but one can check that the argument carries over). 

If $p>p_c^{\smallOTSP}$ then $\nu(|\mc{C}_x^{\smallOTSP}|=\infty)>0$ for each $x$, and so  $\nu(|\mc{B}_x^{\smallOTSP}|=\infty)>0$.  Since the events $|\mc{B}_x^{\smallOTSP}|=\infty$ and $|\mc{C}_x^{\smallOTSP}|=\infty$ depend on disjoint sets of sites (or by the FKG inequality) we also have $\nu(|\mc{C}_x^{\smallOTSP}|=\infty=|\mc{B}_x^{\smallOTSP}|)>0$.  So when $p>p_c^{\smallOTSP}$, ubf's exist with positive probability, and so claim (b) holds by Proposition \ref{PRP:trichotomy}.  By symmetry, lbf's exist if $p<1-p_c^{\smallOTSP}$ and so (c) holds.  Claim (a) now also follows from Proposition \ref{PRP:trichotomy} since in this case $\mc{B}_x$ is neither blocked above nor below. 

\end{proof}

Now consider the 
site percolation model $\FSOSP$, where an open vertex $x\in\Z^2$ connects to 5 neighbours $x-e_1$, $x+e_2$, $x-e_1+e_2$, $x-e_2$, and $x-e_1-e_2$.
Denote the critical $p$ for this model by 
$p_c^{\smallFSOSP}$.
We have not found numerical estimates for $p_c^{\smallFSOSP}$ in the literature, and the best bounds available seem to be
\begin{equation}
\label{FSOSPbounds}
0.3205\le p_c^{\smallFSOSP}\le 0.7491;
\end{equation}
One obtains the upper bound from the inequalities $p_c^{\smallFSOSP}\le p_c^{\smallOTSP}\le p_c^{\smallNE}$ and the bound of Balister et al \cite{BBS} on the latter. The lower bound comes from the inequality $p_c^{\smallFSOSP}\ge 1-p_c^{\smallNSEWalt}$ (a consequence of Corollary \ref{COR:othercasesforB} below, or obtained from the duality between the square lattice and the next-nearest-neighbour square lattice -- see Russo \cite{Russo}), and the upper bound on $p_c^{\smallNSEWalt}$ of Wierman \cite{Wierman95}. We will improve on the lower bound in Theorem~\ref{prp:infiniteM2}, where we show that $p_c^{\smallFSOSP}\ge0.4311$.

\begin{THM}
\label{thm:infiniteB5}
For the model $(\SWE \uparrow)$ with $p\in (0,1)$ we have $\theta_-\in (0,1)$, and $\nu$-a.s.
\begin{enumerate}
\item If $p\ge 1-p_c^{\smallFSOSP}$ and $\mc{B}_o$ is infinite, then $\mc{B}_o=\Z^2$.
\item If $p<1-p_c^{\smallFSOSP}$ and $\mc{B}_o$ is infinite, then there exists a ubf $W$ such that $\mc{B}_o=W_{\le}$.
\end{enumerate}
\end{THM}

\begin{proof}
That $\theta_-\in (0,1)$ is contained in Corollary \ref{cor:infiniteB2}. Next, observe that the hypotheses of Corollary \ref{COR:trichotomy} hold. Following the previous argument, we let $w:\Z\to\Z$ be any function (not necessarily monotone; see Figure \ref{fig:WSE_N_B_o}), and consider when it can be an upper blocking function. Again, vertices in $w_\le^c$ which border $w_\le$ now can lie immediately above, to the right, or to the left of points in $w_\le$. Such vertices can again be enumerated to form a sequence, though this time there can be repetition. More precisely, the possible transitions in this sequence are as follows. 
\begin{itemize}
\item Upwards, e.g.~from $(n,k)$ to $(n,k+1)$. This happens if $w(n)<k<w(n-1)$. 
\item Downwards, e.g.~from $(n,k)$ to $(n,k-1)$. This happens if $w(n+1)>k>w(n)$. 
\item Leftwards, e.g.~from $(n,k)$ to $(n-1,k)$. This happens if $w(n)=w(n-1)=k-1$. 
\item Diagonally NW, e.g.~from $(n,k)$ to $(n-1,k+1)$. This happens if $w(n)<k=w(n-1)$. 
\item Diagonally SW, e.g.~from $(n,k)$ to $(n-1, k-1)$. This happens if $w(n-1)<w(n)=k-1$. 
\end{itemize}
We recognize these as the 5 connections from the vertex $(n,k)$ (if that vertex is open) in the $\FSOSP$ model.  
For $w(n)$ to be an upper blocking function, it is necessary and sufficient that each vertex in this sequence have local environment $\uparrow$, since any arrow of $\SWE$ would give an open path from $w_>$ into $w_{\le}$. Calling $\uparrow$ vertices ``open'' and $\SWE$ vertices ``closed'', we have established the same kind of duality relation as before: upper blocking functions for our random environment correspond precisely to doubly infinite oriented paths in percolation clusters for $\FSOSP$.  

Clearly if $1-p<p_c^{\smallFSOSP}$, no such doubly infinite $\FSOSP$ percolation-paths of $\uparrow$ sites exist. If $1-p=p_c^{\smallFSOSP}$ then Theorem 1 of Grimmett and Hiemer \cite{GH02} can once more be extended to our lattice, showing that percolation clusters for $\FSOSP$ are finite. This establishes (a).  

If $1-p>p_c^{\smallFSOSP}$ then $\nu(|\mc{C}^{\FSOSP}_o|=\infty)>0$, so by Lemma \ref{lem:B_oC_o} (for the  lattice $\FSOSP$) and the FKG inequality, $\vphantom{{|^|}^|}\nu(|\mc{B}^{\FSOSP}_o|=\infty=|\mc{C}^{\FSOSP}_o|)>0$.  Thus ubf's exist with positive probability, and therefore by Corollary \ref{COR:trichotomy} ubf's exist almost surely.  This establishes (b). 
\end{proof}

Of the 2-valued models in $d=2$, which ones exhibit the kind of phase transitions of $\mc{B}_x$ that we have been examining? In other words, when does $\mc{B}_x=\Z^2$ happen for some $p$ but not others? We have seen two such models already: $(\NE,\SW)$ and $(\SWE,\uparrow)$. It turns out that there are precisely three more (modulo rotations and reflections) -- see Table \ref{tab:connections}. The following result describes what happens for each of them. We don't give a proof, since the arguments are simple modifications of ones given already. In each case, $p$ is the probability of the first listed local configuration. Note that in these models it is not possible for $\mc{B}_o$ to be finite, since (\ref{BhasAhalfline}) shows that each $\mc{B}_o$ contains a half line. Thus the alternatives are being $\Z^2$ or being blocked. 

\begin{COR}
\label{COR:othercasesforB}
For $d=2$, the following models have phase transitions as shown.
\begin{enumerate}
\item $(\SWE\NE)$: $\nu(\mc{B}_o=\Z^2)=1$ when $1>p\ge 1-p_c^{\smallOTSP}$; $\mc{B}_o$ is blocked above when $p<1-p_c^{\smallOTSP}$.
\item $(\NSEWalt\NE)$: $\nu(\mc{B}_o=\Z^2)=1$ when $1\ge p\ge 1-p_c^{\smallOTSP}$; $\mc{B}_o$ is blocked above when $p<1-p_c^{\smallOTSP}$.
\item $(\NSEWalt\uparrow)$: $\nu(\mc{B}_o=\Z^2)=1$ when $1\ge p\ge 1-p_c^{\smallFSOSP}$; $\mc{B}_o$ is blocked above when $p<1-p_c^{\smallFSOSP}$.
\end{enumerate}
\end{COR}
Note that the models (b) and (c) above are monotone models since one configuration is a subset of the other.

One can prove duality-type results analogous to those in this section, for the sets $\mc{C}_x$, as well as results about the asymptotic shape of $\mc{B}_x$ or $\mc{C}_x$ when these are blocked above or below. We hope to include them in a subsequent paper.

\section{The communicating clusters $\mc{M}_x=\mc{C}_x \cap \mc{B}_x$}
\label{sec:M_x}
In this section we examine the sets of points that communicate. In many cases, the following trivial lemma immediately shows that $\theta=0$. 
\begin{LEM}
\label{lem:MfiniteTrivially}
Suppose there is some $e$ such that $\mu(\{A:e\in A\})>0$ but $\mu(\{A:-e\in A\})=0$. Then $\mc{M}_x\subset \{y:(y-x)\cdot e = 0\}$ a.s., for every $x$. 
\end{LEM}

This shows that $\theta=0$ in such cases as models $(\WE\,\uparrow)$ and  
$(\NE\NW)$.
We will now give a non-trivial condition on $\mu$ which guarantees that $\mE[|\mc{M}_o|]<\infty$, and hence that $\theta=0$.  Let $\sigma_d$ denote the self-avoiding walk connective constant in $d$ dimensions, defined as $\sigma_d:=\lim_{N\ra \infty} c_N^{1/N}$,
where $c_N=c_N(d)$ is the number of self-avoiding walks of length $N$ (see e.g.~\cite{MadS}).  See Table \ref{tab:SAWcconst} for values of $\sigma_d$ for $d=2,3,4,5$.  In what follows, $\|x\|_1=\sum_{i=1}^d|x^{[i]}|$ for $x\in \Z^d$.
\begin{table}
\begin{center}
\begin{tabular}{c|c|c}
$d$ & rigorous & estimate\\
\hline
2 & $2.6256\le \sigma_2\le 2.6792$ & 2.63816\\
3& $4.5721\le \sigma_3\le 4.7114$& 4.68404\\
4& $6.7429\le \sigma_4\le 6.8040$ & 6.77404\\
5& $8.8285\le \sigma_5\le 8.8602$ & 8.83854
\end{tabular}
\end{center}
\caption{Rigorous lower and upper bounds and numerical estimates for the self-avoiding walk connective constant $\sigma_d$ for $d=2,3,4,5$ (square lattice), see Finch \cite[Table 5.2]{Finch} and Jensen \cite{J04}.}
\label{tab:SAWcconst}
\end{table}

\begin{THM}
\label{lem:Mfinite}
Fix $d\ge 2$, and suppose that there exists a set $V$ of mutually orthogonal unit vectors such that $\mu(\{A:\emptyset\ne A\subset V\})\ge 1-\epsilon$.  Then $\mE[|\mc{M}_o|]<\infty$ if $\epsilon <\sigma_d^{-2}$.
\end{THM}
\proof 

Without loss of generality, assume that $V=\{e_1,\dots, e_k\}$ for some $k\le d$, and let $v=\sum_{u\in V}u$.  
Let $L^+$ (resp. $L^-$) be the set of points $x\in \Z^d$ such that $x\cdot v\ge 0$ (resp. $\le 0$).  Let $L^+_N=L^+\cap \{x:\|x\|_1=N\}$ and similarly for $L_N^-$.  Note that $\{x:\|x\|_1<r\}\cup \bigcup_{N=r}^{\infty}(L^+_N\cup L_N^-)=\Z^d$.

Suppose that $\mu(\{A:\emptyset\ne A\subset V\})\ge 1-\epsilon$.  We claim that for $\epsilon<\sigma_d^{-2}$ there exist $c>0$ and $\eta\in (0,1)$ such that for all $N$, 
\begin{equation}
\nu(\mc{C}_o\cap L_N^-\ne \emptyset)<c\eta^N.
\label{claimoflem:Mfinite}
\end{equation}  
Assume for the moment that this is true. 
For $x\in L_N^-$, $\nu(x\in\mc{C}_o)\le \nu(\mc{C}_o\cap L_N^-\ne \emptyset)$. Likewise if $x\in L_N^+ $ then by translation invariance,
$\nu(o\in\mc{C}_x)=\nu(-x\in\mc{C}_o)\le \nu(\mc{C}_o\cap L_N^-\ne \emptyset)$. 
For $r\gg 1$, this implies that 
\eqalign
\mE[|\mc{M}_o|]=&\mE\left[\sum_{x\in \Z^d} I_{\{x \in \mc{C}_o\}}I_{\{o \in \mc{C}_x\}}\right]\le \sum_{x\in \Z^d}\nu(x \in \mc{C}_o)^{\hlf}\nu(o \in \mc{C}_x)^{\hlf}\nn\\
\le&\sum_{x:\|x\|_1< r}1+\sum_{N=r}^{\infty}\left(\sum_{x\in L_N^-}\nu(x \in \mc{C}_o)^{\hlf}+\sum_{x\in L_N^+}\nu(o \in \mc{C}_x)^{\hlf}\right)\nn\\
\le &C_d r^d+c'\sum_{N=r}^{\infty}\left(\sum_{x\in L_N^-}\eta^{\frac{N}{2}}+\sum_{x\in L_N^+}\eta^{\frac{N}{2}}\right)\le C_d r^d+\sum_{N=r}^{\infty}C_d'N^{d-1}\eta^{\frac{N}{2}}<\infty.\nn
\enalign
Hence it only remains to verify the claim of (\ref{claimoflem:Mfinite}).

If $\mc{C}_o\cap L_N^-\ne \emptyset$ then there is an open self-avoiding path from $o$ to some $x$ such that $\sum_{i=1}^kx_i\le 0$ and $\|x\|_1=N$.  Moreover any finite path connecting $o$ to $x$ must consist of at least $N$ steps, with at least half of the steps of the path being taken in the directions taken from $\{e_1,\dots,e_k\}^c$ (if more than half of the steps of a finite path are taken from $\{e_1,\dots,e_k\}$, then the endpoint $y$ of the path must have $\sum_{i=1}^ky_i>0$).  

The probability that at least one of the $c_N$ self-avoiding paths of length $N$ is an open path, at least half of whose steps are in the directions $\{e_1,\dots,e_k\}^c$ is at most $c_N$ times the maximum (over paths) probability that a particular such path is open.  
Any such path has at least $\lfloor N/2\rfloor$ steps in $\{e_1,\dots,e_k\}^c$, each of which is open with probability at most $\epsilon$.  
It follows that the probability that there is an open self-avoiding path of length $N$ starting from the origin, at least half of whose steps are in the directions $\{e_{k+1},\dots,e_d\}$ is at most 
\[
c_N\epsilon^{\floor{\frac{N}{2}}}\le C\eta^N,
\]
provided that $\sqrt{\epsilon}\sigma_d<\eta$. Accordingly, if $\epsilon<\sigma^{-2}_d$, we can find such an $\eta<1$.
\qed

Note that we can improve on the bound $\sigma_d^{-2}$ with more information about the measure $\mu$. For example, suppose in addition to the hypotheses of Theorem \ref{lem:Mfinite}, we have a dichotomy
$$
\mu(\{A: A\subset V\text{ or }A\subset V^c\})=1.
$$
Then as long as $\epsilon<1/2$, the bound on the relevant probability becomes
\[
c_N\epsilon^{\floor{\frac{N}{2}}}(1-\epsilon)^{N-\floor{\frac{N}{2}}},
\]
and solving the quadratic inequality, we find an $\eta<1$ giving the conclusion of the theorem, provided that
$$
\epsilon<\epsilon_d^0=\hlf \left(1-\sqrt{1-\frac{4}{\sigma_d^2}}\right).
$$
For $d=2$, the rigorous upper bound on $\sigma_2$ gives values $\sigma_2^{-2}=0.13931$ and $\epsilon_2^0=0.16730$

This immediately implies the following. 
\begin{COR} All $\mc{M}_x$ are finite 
\label{cor:Mfinite_2d_orthant}
in the model $(\NE \SW)$ 
whenever $p>1-\epsilon_2^0=0.83270$ or $p<\epsilon_2^0=0.16730$, and in the model
$(\SWE\uparrow)$ 
whenever $p<\epsilon_2^0=0.16730$.
\end{COR}
For further 2-dimensional examples, see Table \ref{tab:connections}. For $d\ge 2$ we call the model $(\mc{E}_+\,\, \mc{E}_-)$ the {\it orthant model} (so $(\NE\SW)$ is the case $d=2$). For this model we have that 
 $\mc{M}_x$ is a.s.~finite whenever $p<\epsilon_d^0$ or  $p>1-\epsilon_d^0$.


For some models we can instead show that there are infinite mutually-connected clusters.  The following two lemmas are trivial and apply e.g.~to models $(\SWE\WE)$ and $(\SWE\NWE) $ respectively.  
\begin{LEM}
Suppose there is an $e$ such that $\mu(\{A:e\in A, -e\in A\})=1$. Then $\theta=1$.
\end{LEM}
\begin{LEM}
\label{lem:MisZd}
Suppose that for every $e$, $\mu(\{A:e\in A\})>0$, and that for some $\tilde e$, $\mu(\{A:\tilde e\in A, -\tilde e\in A\})=1$. Then $\mc{M}_x=\Z^d$ for every $x$.
\end{LEM}

In the more interesting cases, which we now turn to, there will be a unique infinite $\mc{M}_x$, and infinitely many finite $\mc{M}_y$. We start with the following stronger definition.

\begin{DEF}
\label{def:gigantic}
We say that  $\mc{G}$ has a {\em gigantic $\mc{M}$-component} if there is an $x$ such that $\mc{M}_x$ is infinite and all $\Z^d$-connected components of $\Z^d\setminus\mc{M}_x$ are finite.
\end{DEF}

\begin{LEM}
\label{lem:giganticimplications}
Assume that $\theta_+=1$ and that $\mc{G}$ has a gigantic $\mc{M}$-component $\mc{M}_x$. Then $\nu$-a.s.,
\begin{enumerate}
\item $\mc{M}_x$ is the only infinite equivalence class for the communication relation,
\item $\mc{C}_y\supset\mc{M}_{x}$ for every $y$, so all $\mc{C}_y$ intersect, and $\mc{C}_y=\mc{M}_x$ for $y\in\mc{M}_x$,  
\item $\mc{B}_y=\Z^d$ for $y\in\mc{M}_{x}$, and $\mc{B}_y$ is finite otherwise. 
\end{enumerate}
\end{LEM}
\begin{proof} (a) is immediate, since each $\mc{M}_y$ is connected, and either coincides with $\mc{M}_x$ or is disjoint from it. In the latter case, it is therefore contained in a component of $\Z^d\setminus\mc{M}_x$, which is finite by hypothesis. Since all $\mc{C}_y$ are infinite, they leave finite components of $\Z^d\setminus\mc{M}_x$, and hence intersect $\mc{M}_x$.  This implies that $\mc{C}_y\supset\mc{M}_{x}$ for every $y$ and that $\mc{B}_y=\Z^d$ for $y\in\mc{M}_{x}$ and that $\mc{B}_y\subset \Z^d\setminus\mc{M}_x$ is finite otherwise.  Finally, since  $\mc{C}_z\supset\mc{M}_{x}$ for any $z$, if $y\in\mc{M}_{x}$ then $y\in \mc{C}_z$ for each $z\in \mc{C}_y$ so $\mc{C}_y=\mc{M}_y=\mc{M}_x$ if $y\in\mc{M}_{x}$.  
\end{proof}

Under the assumption $\theta_+=1$, we believe that having $\mc{B}_y=\Z^2$ in 2 dimensions is equivalent to having $\mc{M}_y$ be a gigantic $\mc{M}$-component. 
The following theorem verifies this under additional regularity conditions on $\mu$.

\begin{THM}
\label{BvsM}
Assume the hypotheses of Proposition \ref{PRP:trichotomy} or Corollary \ref{COR:trichotomy}. There is $\nu$-almost surely a gigantic $\mc{M}$-component 
$\Longleftrightarrow \nu(\mc{B}_o=\Z^2)>0$.
\end{THM}
\begin{proof}
Note that the hypotheses of Proposition \ref{PRP:trichotomy} or Corollary \ref{COR:trichotomy} imply $\theta_+=1$.
So the $\Rightarrow$ direction just reiterates (c) of Lemma \ref{lem:giganticimplications}, and we must prove $\Leftarrow$. 

First, assume the hypotheses of Corollary \ref{COR:trichotomy}. Suppose $\{\mc{B}_o=\Z^2\}$ has positive probability. Then $\nu$-a.s., $\mc{B}_y$ is infinite for some $y$ by Lemma \ref{lem:exist_inf_B}, and by Corollary \ref{COR:trichotomy}, $\mc{B}_y=\Z^2$. Thus, $\mc{M}_y=\mc{C}_y$ is infinite a.s.~since $\theta_+=1$.  Likewise any other $z$ with $\mc{M}_z$ infinite automatically has $\mc{B}_z=\Z^2$, and hence $\mc{M}_z=\mc{M}_y$ (as we can get $y\leftrightarrow z$). Thus there is only one infinite equivalence class for the communication relation, which we denote by $\mc{M}$. It remains to show that $\mc{M}$ is gigantic, that is, that connected components of $\Z^2\setminus\mc{M}$ are finite.  


Suppose $o\notin\mc{M}$. It is sufficient to consider the $\Z^2\setminus\mc{M}$-connected component of $o$. By ergodicity,  find $x_0$ and $x_1$ in $\mc{M}$ with $x_0^{[1]}=x_1^{[1]}=0$ and $x_0^{[2]}<0<x_1^{[2]}$.
Run the NW paths (corresponding to a subnetwork $(\uparrow\,\leftarrow)$) from $x_0$ and $x_1$. By Lemma \ref{lem:coalescence1} these paths meet at some point $x_2$. Run NE paths (corresponding to $(\uparrow\,\rightarrow)$) from $x_0$ and $x_1$ till they meet at a point $x_3$. 
Since $x_0\in\mc{M}$ we know that $\mc{B}_{x_0}\supset\mc{M}_{x_0}=\mc{M}$ is infinite, and therefore $\mc{B}_{x_0}=\Z^2$. 
So $\mc{M}=\mc{M}_{x_0}=\mc{C}_{x_0}$. Likewise $\mc{M}=\mc{M}_{x_1}=\mc{C}_{x_1}$.
The four paths $x_0\to x_2$, $x_1\to x_2$, $x_0\to x_3$, and $x_1\to x_3$ 
therefore all lie in $\mc{M}$ and enclose $o$ between them. Since they enclose $o$, they also enclose the connected component of 
$\Z^2\setminus\mc{M}$ which contains $o$. Thus this component is finite, as required.

A similar argument works if we assume the hypotheses of Proposition \ref{PRP:trichotomy} instead. 
\end{proof}\medskip

Combining Theorem \ref{BvsM} with Corollary \ref{COR:B_o_NS_EW}, Corollary \ref{COR:othercasesforB}, and Theorems \ref{thm:infiniteBorthant} and \ref{thm:infiniteB5}, we obtain the following (see Theorems \ref{prp:infiniteM2} and \ref{prp:infiniteM3} for estimates of the critical probabilities herein).
\begin{THM}
\label{prp:infiniteM1}
$\mc{G}$ has a gigantic $\mc{M}$-component almost surely [resp. with probability zero], in the following cases:
\begin{enumerate}
\item the model $(\WE\,\NS)$ with $0<p<1$; 
\item the model $(\NE\SW)$ with $1-p_c^{\smallOTSP}\le p\le p_c^{\smallOTSP}$ \quad[resp. $0\le p<1-p_c^{\smallOTSP}$ or $p_c^{\smallOTSP}<p\le 1$] 
\item the models $(\SWE\NE)$ and $(\NSEWalt\NE)$ with $1-p_c^{\smallOTSP}\le p<1$ \quad[resp. $0\le p<1-p_c^{\smallOTSP}$] 
\item the models $(\SWE\uparrow)$ and $(\NSEWalt\uparrow)$ with $1-p_c^{\smallFSOSP}\le p<1$ \quad[resp. $0\le p<1-p_c^{\smallFSOSP}$] 
\end{enumerate}
\end{THM}
Note that the regularity hypotheses of Corollary \ref{COR:trichotomy} by themselves do not rule out the possibility that there exist infinite $\mc{M}_x$ which fail to be gigantic components. For example, the model $(\SWE\,\WE)$ satisfies the hypotheses of Corollary \ref{COR:trichotomy} if we rotate it by $180^\circ$, but for each $x$, $\mc{M}_x$ is the horizontal line containing $x$. The techniques of the current paper also do not resolve the question of whether there can be multiple infinite (but not gigantic) $\mc{M}_x$ clusters for the model $(\NE\SW)$. This would follow if $\theta(p)=0$ whenever $0<p<1-p_c^{\smallOTSP}$ or $p_c^{\smallOTSP}<p<1$. We have another argument which we believe proves the latter, and hope to include it in a subsequent paper. 


The rest of this section will be spent giving elementary arguments that gigantic $\mc{M}$ components exist in several models, for $p$'s in various concrete intervals. Because of the structural results derived earlier, this automatically provides bounds on the critical percolation values $p_c^{\smallOTSP}$ and $p_c^{\smallFSOSP}$. As far as we are aware, these bounds improve on what is in the literature (see (\ref{OTSPbounds}) and (\ref{FSOSPbounds})), but the main purpose is to understand what can be deduced in an elementary way using the approach through degenerate environments. 


We will use the notion of an {\sl open cycle}, by which we mean a set of vertices $R$ which can be enumerated as a closed path $x_0x_1x_2\dots x_N$ such that $x_0=x_N$ and $x_{n+1}-x_n\in\mc{G}_{x_n}$ for $0\le n<N$. The idea will be to construct, for each $M\ge 1$, an open cycle $R_M$ that encloses the box $[-M,M]^2$. If $\mc{B}_x$ is infinite, then whenever $M>|x^{[1]}|\lor|x^{[2]}|$ we have $x\in[-M,M]^2$, so $\mc{B}_x$ intersects $R_M$. This implies that $R_M\subset\mc{B}_x$. Likewise $\mc{C}_x$ is infinite (assuming (\ref{standinghypothesis})), so intersects $R_M$, from which it follows that $R_M\subset\mc{C}_x$. Thus $R_M\subset\mc{M}_x$ for all sufficiently large $M$, implying that $\mc{M}_x$ is infinite. Moreover, the construction implies that all connected components of $\Z^2\setminus\mc{M}_x$ are contained within some $R_M$, so must be finite. In other words, we have proved the following:
\begin{LEM}
\label{lem:cycleing}
Assume $d=2$ and that $\theta_+=1$. Assume further that open cycles $R_M$ can be constructed enclosing arbitrarily large boxes $[-M,M]^2$. Then $\mc{M}_x$ is a gigantic $\mc{M}$ component whenever $\mc{B}_x$ is infinite. 
\end{LEM}

An easy application is to give an alternate proof of the existence of a gigantic $\mc{M}$ component in the $(\WE\,\NS)$ model:
\begin{proof}[Alternate proof of Theorem \ref{prp:infiniteM1}]
For fixed $M\ge 1$ we can build an outward spiral path from the following pieces. 
\begin{itemize}
\item Follow a SW path from $y_0=(-M,M)$ till it reaches a point $y_1$ with $y_1^{[2]}=-M$. 
\item Then follow a SE path till it reaches a $y_2$ with $y_2^{[1]}=M$. 
\item Then follow a NE path till it reaches a $y_3$ with $y_3^{[2]}=M$. 
\item Then follow a NW path till it reaches a $y_4$ with $y_4^{[1]}=-M$. 
\item Then follow a SW path till it reaches a $y_5$ with $y_5^{[2]}=-M$. 
\item Then follow a SE path till it reaches a $y_6$ with $y_6^{[1]}=M$. 
\item Then follow a NE path till it reaches a $y_7$ with $y_7^{[2]}=M$. 
\end{itemize}
It is possible that for some large $M$, this spiral closes in on itself, in which case we've produced the desired open cycle. If it doesn't, then by Proposition \ref{prp:infiniteB1}, for all $M$ sufficiently large, there exists $x\in [-M,M]^2$ with $\mc{B}_x$ infinite for the submodel $(\WE\,\uparrow)$.   By Proposition \ref{prp:infiniteB1}(b)
there is an infinite open path to $x$, which must cross our spiral at some point $z$  between $y_4$ and $y_7$, and then again at some $z'$ between $y_0$ and $y_3$. Following the spiral from $z'$ around to $z$, and then moving from $z$ to $z'$ along the path in $\mc{B}_x$ produces the cycle $R_M$. Now apply Lemma \ref{lem:cycleing}.
\end{proof}

As remarked below (\ref{FSOSPbounds}), the following theorem improves on the existing bound $p_c^{\smallFSOSP}\ge 0.3205$.

\begin{THM}
\label{prp:infiniteM2}
For the model $(\SWE\uparrow)$, if $p^3-p^2+2p-1>0$ and $p<1$ then $\mc{G}$ has a gigantic $\mc{M}$-component, almost surely. In consequence, $p_c^{\smallFSOSP}\ge 0.4311$.
\end{THM}
\proof
We will again show that there is a gigantic $\mc{M}$-component by constructing arbitrarily large cycles $R_M$. 
Again our network contains the network $(\leftrightarrow\,\uparrow)$ 
so by Lemma \ref{lem:exist_inf_B} and Proposition \ref{prp:infiniteB1} 
there is an $x$ such that $\mc{B}_x$ contains a semi-infinite path $\dots x_2x_1x_0$ to $x$ such that $x^{[2]}$ is monotone. Without loss of generality, we take this $x=o$. 

As in the alternate proof of Theorem \ref{prp:infiniteM1}, we may construct NE and NW paths in our network from arbitrary initial vertices. But we will also need paths that play the role of SE or SW paths in that proof. Here these will be paths that only move SE on average, or SW on average. So define the {\sl SEoA} path from a $\SWE$ vertex to go $\downarrow$ if this leads to another $\SWE$ vertex, and otherwise to go $\rightarrow$. From a $\uparrow$ vertex, the path of course goes $\uparrow$. We define a {\sl SWoA} path similarly.  By construction, the {\sl SEoA} path never takes a west step and the {\sl SWoA} path never takes an east step. Our first task is to see when these paths actually do -- on average -- move in the desired direction.

Let $z$ be the initial point of the SEoA path, and let $W$ be the vertical distance travelled before moving sideways. So $W=-k\le 0$ means that $z$ as well as the $k$ points directly below it are $\SWE$, while the point below them is $\uparrow$. Likewise, $W=k\ge 1$ means that $z$ as well as the $k-1$ points directly above it are $\uparrow$, while the point above them is $\SWE$. Thus
$$
\mE[W]=-\sum_{k\ge 0}kp^{k+1}(1-p)+\sum_{k\ge 1}kp(1-p)^k=-\frac{p^2}{1-p}+\frac{1-p}{p}=-\frac{p^3-p^2+2p-1}{p(1-p)}.
$$
In particular, if $p^3-p^2+2p-1>0$, then the SEoA paths (resp. SWoA paths) drift SE (resp. SW) on average. 

So assume $p^3-p^2+2p-1>0$, and construct the cycle as follows. Let $z^i=z^i_0 z^i_1\dots$ be the SEoA path 
starting from $(-i,M)$ and ending at $(M,k)$ for some $k$. Note that two such paths coalesce if/when they meet. Because $\mE[W]<0$, the 
probability that $[-M,M]^2$ lies entirely above this path converges to 1 as $i\to\infty$. So we may almost 
surely find an $i$ so that this is this the case. In fact, we may find an increasing sequence $i(0), i(1), \dots$ 
such that $z^{i(k)}$ lies entirely above $z^{i(k+1)}$ for $k\ge 0$, and $[-M,M]^2$ lies entirely above each $z^{i(k)}$. 

We may now build a spiral path as follows: 
\begin{itemize}
\item From $y_1=z^{i(0)}_0$ (which has $y^{[2]}_1=M$), follow the SEoA path till it reaches a point $y_2$ with $y_2^{[1]}=M$. By construction, $[-M,M]^2$ lies above this path.
\item From $y_2$ follow the NE path till it reaches a point $y_3$ with $y_3^{[2]}=M$.
\item From $y_3$ follow the NW path till it reaches a point $y_4$ with $y_4^{[1]}=-M$.
\item From $y_4$ follow the SWoA path till it hits some path $z^{i(k)}$ at a point $y_5$. It must do so because eventually it lies below the line $x^{[2]}=M$, so $z^{i(k)}$ will cross it for $k$ sufficiently large. 
\item From $y_5$ follow the SEoA path $z^{i(k)}$ till it reaches a point $y_6$ with $y_6^{[1]}=M$.
\item From $y_6$ follow the NE path till it reaches a point $y_7$ with $y_7^{[2]}=M$.
\end{itemize}
It is possible that this spiral closes in on itself, in which case we've produced the desired cycle. If it doesn't, recall that we have an infinite path $\dots x_2x_1x_0$ leading to $o$ in $\mc{B}_o$, with $x_i^{[2]}$ monotone decreasing. This must cross our spiral at some point $x'$  between $y_4$ and $y_7$, and then again at some point $x''$ between $y_1$ and $y_3$. Following the spiral from $x''$ around to $x'$, and then moving from $x'$ to $x''$ along the path in $\mc{B}_o$ produces the cycle $R_M$. Now apply Lemma \ref{lem:cycleing}.

The estimate on $p_c^{\smallFSOSP}$ comes from computing the unique root of the increasing function $p^3-p^2+2p-1$.  By Lemma \ref{lem:giganticimplications} we will have some $\mc{B}_x=\Z^2$ for $p$ above this root, and can then apply Theorem \ref{thm:infiniteB5}.
\qed

As remarked below (\ref{OTSPbounds}), the following theorem improves on the existing bound $p_c^{\smallOTSP}\ge 0.5$ and is the first result we are aware of separating oriented from non-oriented percolation on the triangular lattice.

\begin{THM}
\label{prp:infiniteM3}
For the model $(\NE \SW)$, if $p^3+2p-1\ge 0$ and $(1-p)^3+2(1-p)-1\ge 0$ then $\mc{G}$ has a gigantic $\mc{M}$-component almost surely. In consequence $p_c^{\smallOTSP}\ge 0.5466$.
\end{THM}
\proof
This network includes the network $(\NE,\leftarrow)$ 
so by Proposition \ref{prp:infiniteB2} there 
are infinite $\mc{B}$'s. In fact, let $\mc{N}_y\subset\mc{B}_y$ denote the cluster of points from which $y$ can be reached, using only steps $\uparrow, \leftarrow, \rightarrow$. Let $L_n(y)$ and $U_n(y)$ be the infimum and supremum of $l$ with $(l,y^{[2]}-n)\in\mc{N}_y$. It follows from Proposition \ref{prp:infiniteB2} that $A_y=\{|L_n(y)|<\infty \, \forall \,n\}$ has probability $>0$. Without loss of generality we (for now) take $y=o$ and write $L_n=L_n(o)$. 

Consider the SE path from some point $z$ lying to the left of $\mc{N}_o$. Our first task will be to determine what choices of $p$ imply that, a.s. on the event $A_o$, this path hits $\mc{N}_o$. Without loss of generality, $z=(w,0)$ for some $w<L_0$, and $\mc{G}_z=\SW$.

Recall that on the event $A_o$, $L_n$ agrees with the path of a random walk $L_n^0$. We may construct $L^0_n$ as follows: from $(L_n^0,-n)$ look down one vertex. If $\mc{G}_{(L_n^0,-(n+1))}=\NE$ then $L_{n+1}^0\le L_n^0$, and is in fact the smallest value such that $\mc{G}_{(j,-(n+1))}=\NE$ for $L_{n+1}^0\le j\le L_n^0$. On the other hand, if  $\mc{G}_{(L_n^0,-(n+1))}=\leftarrow$, then $L_{n+1}^0>L_n^0$, and is in fact the smallest $j>L_n^0$ such that $\mc{G}_{(j,-(n+1))}=\NE$. As observed in the proof of Proposition \ref{prp:infiniteB2}, 
$$
\nu(L_{n+1}^0=l+j\mid L_n^0=l)=
\begin{cases}
(1-p)p^{1-j}, &j\le 0\\
p(1-p)^j,&j>0
\end{cases},
\qquad
\mE[\Delta L_n^0]= \frac{1-p}{p}-\frac{p^2}{1-p}.
$$
Set $W_0=w$. Let $W_1, W_2, \dots$ be the first coordinates of successive vertices at which the SE path moves downward. In other words, the downwards steps are from $(W_n, -n)$ to $(W_n, -(n+1))$. Our object is to show that with probability 1 there exists an $n$ with $W_n\ge L_n^0$. 

$W_n$ is itself a random walk, with $\nu(W_{n+1}=l+j\mid W_n=l)=(1-p)p^j$ for $j\ge 0$, and $\mE[\Delta W_n]=p/(1-p)$. We see that $\mE[\Delta W_n]\ge \mE[\Delta L_n^0]$ provided
$$
\frac{p}{1-p}\ge \frac{1-p}{p}-\frac{p^2}{1-p}\quad
\Longleftrightarrow\quad
p^3+2p-1\ge 0.
$$
We assume, in what follows, that this inequality holds. 
In particular, this is true for $p\ge 0.4534$; If $W_n$ and $L^0_n$ were independent, the desired conclusion would follow immediately. Because these walks are actually not quite independent, we need to be slightly more careful. 

When $W_n=j<l=L_n^0$, we have
$$
\nu(W_{n+1}=j', L_{n+1}^0=l'\mid W_n=j, L_n^0=l)=\nu(W_{n+1}=j'\mid W_n=j)\cdot \nu(L_{n+1}^0=l'\mid L_n^0=l),
$$
provided $j'< l$ and $j'< l'-1$, since the two events in question depend on disjoint parts of the environment. Let $C(j,l)$ be the set of $(j',l')$ which violate the above condition. Then we can describe the evolution of the Markov chain $(W_n,L_n^0)$ as follows: from $(W_n,L_n^0)=(j,l)$ propose a move to a $(j',l')$ chosen based on $W_{n+1}$ and $L_{n+1}^0$ evolving independently. If $(j',l')\notin C(j,l)$ the move is accepted. Otherwise the move is rejected, and replaced by a move to some point of $C(j,l)$ chosen according to the required law. The fact that $\mE[\Delta W_n]\ge \mE[\Delta L_n^0]$ implies that with probability 1, this chain will eventually encounter a rejected move. 

What moves in $C(j,l)$ can replace a rejected move? There are three types -- either $j'\ge l'$ (if $\mc{G}_{(i,-(n+1))}=\NE$ for $j\le i\le l$), or $j'+1=l'\le l$ (if $\mc{G}_{(l,-(n+1))}=\NE$ and exactly one $j\le i<l$ has $\mc{G}_{(i,-(n+1))}=\SW$), or $j'=l<l'$ (if $\mc{G}_{(l,-(n+1))}=\SW$ and all $j\le i<l$ have $\mc{G}_{(i,-(n+1))}=\NE$). In the first case, $W_{n+1}\ge L_{n+1}^0$ and the paths cross. In the second case, there is a high probability of crossing on the next step. It is then an easy calculation to show that there is an $\epsilon>0$ such that 
$$
\nu(\text{$W_{n+1}\ge L_{n+1}^0$ or $W_{n+2}\ge L_{n+2}^0$}\mid W_n=j, L_n^0=l)\ge\epsilon \nu((W_{n+1},L^0_{n+1})\in C(j,l)\mid W_n=j, L_n^0=l).
$$
Thus after at most finitely many rejected moves, we will eventually find one leading to the desired intersection. We have therefore shown that with probability 1, there exists an $n$ with $W_n\ge L_n^0$, as required. 

Suppose that $p^3+2p-1\ge 0$. We claim that with probability 1 there is a $j>M$ such that $\mc{N}_{(j,M)}$ is infinite and $[-M,M]^2$ lies to the left of $\mc{N}_{(j,M)}$. To see this, let $\epsilon=\nu(|\mc{N}_{(j,M)}|=\infty)>0$. Because $\mE[\Delta L_n^0]>-\infty$ there is a $j_0>M$ such that 
$$
\nu(\text{$|\mc{N}_{(j_0,M)}|=\infty$ and $[-M,M]^2$ lies to the left of $\mc{N}_{(j_0,M)}$})\ge\epsilon/2.
$$
Search down from $(j_0,M)$ to see if this event occurs. Either the search succeeds, or it fails after examining a finite number $k_0$ of vertices. If it fails, we may likewise find a $j_1>j_0$ such that 
$$
\nu(\text{$|\mc{N}_{(j_1,M)}|=\infty$ and $[-(M+j_0+k_0),M+j_0+k_0]^2$ is left of $\mc{N}_{(j_1,M)}$})\ge\epsilon/2.
$$
Search down from $(j_1,M)$ to see if this event occurs. Note that this search will not involve any of the $k_0$ vertices already examined, so it succeeds or fails independently of what has come before. Either the search succeeds, or it fails after examining a finite number $k_1$ of vertices, and the process can now be repeated. Eventually the search must succeed, giving the claimed value $j$.

Suppose also that $(1-p)^3+2(1-p)-1\ge 0$. Define $\tilde{\mc{N}}_y$ as above, but using moves $\downarrow, \leftarrow, \rightarrow$. Then by symmetry, a similar argument shows that there is a $\tilde j<-M$ such that $\tilde{\mc{N}}_{(\tilde j,-M)}$ is infinite, and $[-M,M]^2$ lies to the right of $\tilde{\mc{N}}_{(\tilde j,-M)}$. As well, NW paths from the right of $\tilde{\mc{N}}_{(\tilde j,-M)}$ will a.s. intersect this set. 

Finally, we can construct our cycle $R_M$. Starting from $(\tilde j,-M)$ follow the SE path till it intersects $\mc{N}_{(j,M)}$. By definition of 
$\mc{N}_{(j,M)}$ we may then follow a path in this set up to $(j,M)$, and by construction $[-M,M]^2$ lies to the left of this path. Now follow a NW path till it intersects $\tilde{\mc{N}}_{(\tilde j,-M)}$. By definition we may follow a path in this set up to $(\tilde j,-M)$, which closes up the desired cycle. Now apply Lemma \ref{lem:cycleing}. 

The estimate on $p_c^{\smallOTSP}$ comes from computing the unique root $r$ of the increasing function $p^3+2p-1$. We will have some $\mc{B}_x=\Z^2$ for $p$ between $r$ and $1-r$, and can then apply Theorem \ref{thm:infiniteBorthant}.\qed

\medskip


\section{Model summary}
\label{sec:modelsummary}
Here we summarize the results of earlier sections as applied to 2-dimensional 2-valued environments.  In each case the first possibility is assumed to have probability $p\in (0,1)$, and the second $1-p$.

\vfill
\begin{table}
\begin{center}
\begin{tabular}{l|c|c|c|c|c}
Model & $\theta_+$ & $\theta_-$ & $\theta$ & Notes 
\\
\hline
$\uparrow$ $\cdot$ & 0&0&0&\vphantom{${I^I}^I$} 1-dimensional \\
$\leftrightarrow$ $\cdot$ & 0&0&0& 1-dimensional \\
$\NE$ $\cdot$ & $>0$($p>p_c^{\smallNE}$) &  {\sl see (i)} &0& oriented site percolation; {\sl ii}   \\
$\SWE$ $\cdot$ & $>0$($p>p_c^{\smallSWE}$) &  {\sl see (i)} &0& partially-oriented site perc.; {\sl ii} \\
$\NSEWalt$ $\cdot$ & $>0$($p>p_c^{\smallNSEWalt}$) &    {\sl see (i)} & {\sl see (i)} & site percolation;  {\sl ii} \\
\hline
$\uparrow$ $\rightarrow$ & 1 & 0 & 0 &\vphantom{${I^I}^I$} coalescing RW;  {\sl iii} \\
$\uparrow$ $\downarrow$ & 0 & 0 & 0 & 1-dimensional \\
$\leftrightarrow$ $\uparrow$ & 1 & $>0$ & 0 &  {\sl iii, iv, v${}^1$} \\
$\leftrightarrow$ $\rightarrow$ & 1 & 1 & 0 & 1-dimensional \\
$\leftrightarrow$ $\updownarrow$ & 1 & $>0$ & $>0$ & {\sl iii, iv, v${}^2$,vi${}^1$} \\
\hline
$\NE$ $\uparrow$ & 1 & $1$ & $0$ &\vphantom{${I^I}^I$} {\sl iii, iv} \\
$\NE$ $\NW$ & 1 & $1$ & $0$ &  {\sl iii, iv, v${}^1$} \\
$\NE$ $\leftrightarrow$ & 1 & $1$ & $0$ & {\sl iii, iv, v${}^1$} &\\
$\NE$ $\leftarrow$ & 1 & $>0$  & $0$ & {\sl iii, iv} \\
$\NE$ $\SW$ & 1 & $>0$ & phase trans. & {\sl iii, iv, v${}^1$} ($p<1-p_c^{\smallOTSP}$ or $>p_c^{\smallOTSP}$) \\
&&&& {\sl iii, iv, v${}^2$, vi${}^1$} ($1-p_c^{\smallOTSP}\le p\le p_c^{\smallOTSP}$) \\
\hline
$\SWE$ $\downarrow$ & 1 & $1$ & $0$ &\vphantom{${I^I}^I$} {\sl iii, iv, v${}^1$} \\
$\SWE$ $\rightarrow$ & 1 & $1$ & $0$ & {\sl iii, iv} \\
$\SWE$ $\uparrow$ & 1 & $>0$ & phase trans. &  {\sl iii, iv, v${}^1$} ($p<1-p_c^{\smallFSOSP}$) \\
&&&& {\sl iii, iv, v${}^2$, vi${}^1$} ($p\ge 1-p_c^{\smallFSOSP}$) \\
$\SWE$ $\leftrightarrow$ & 1 & $1$ & $1$ & {\sl iii, iv, v${}^1$, vi${}^3$}  \\
$\SWE$ $\updownarrow$ & 1 & $1$ & $1$ & {\sl iii, iv, v${}^2$, vi${}^2$} \\
$\SWE$ $\NE$ & 1 & $1$ & phase trans. & {\sl iii, iv, v${}^1$} ($p<1-p_c^{\smallOTSP}$) \\
&&&& {\sl iii, iv, v${}^2$, vi${}^2$} ($p\ge 1-p_c^{\smallOTSP}$) \\
$\SWE$ $\SW$ & 1 & $1$ & $0$ & {\sl iii, iv, v${}^1$} \\
$\SWE$ $\NSE$ & 1 & $1$ & $1$ & {\sl iii, iv, v${}^2$, vi${}^2$} \\
$\SWE$ $\NWE$ & 1 & $1$ & $1$ &  {\sl iii, iv, v${}^2$, vi${}^2$}  \\
\hline
$\NSEWalt$ $\uparrow$ & 1 & $1$ & phase trans. &\vphantom{${{|^|}^|}^|$} {\sl iii, iv, v${}^1$} ($p<1-p_c^{\smallFSOSP}$) \\
&&&&  {\sl iii, iv, v${}^2$, vi${}^2$} ($p\ge 1-p_c^{\smallFSOSP}$) \\
$\NSEWalt$ $\NE$ & 1 & $1$ & phase trans. &  {\sl iii, iv, v${}^1$} ($p<1-p_c^{\smallOTSP}$) \\
&&&&  {\sl iii, iv, v${}^2$, vi${}^2$} ($p\ge 1-p_c^{\smallOTSP}$)  \\
$\NSEWalt$ $\WE$ & 1 & $1$ & $1$ &  {\sl iii, iv, v${}^2$, vi${}^2$} \\
$\NSEWalt$ $\SWE$ & 1 & $1$ & $1$ & {\sl iii, iv, v${}^2$, vi${}^2$} \\
\hline
\end{tabular}
\end{center}
\caption{Table of connectivity results for 2-dimensional 2-valued degenerate random environments.}
\label{tab:connections}
\end{table}

\noindent {\bf Notes to Table \ref{tab:connections}} 
\begin{enumerate}
\item[(i)] is infinite with probability $>0 \Longleftrightarrow \mc{C}_o$ is.
\item[(ii)] Phase transition: $\mc{C}_o$ is finite for $p<p_c$, and $\infty$ with probability $>0$ if $p>p_c$.
\item[(iii)] All $\mc{C}_x$ intersect.
\item[(iv)] All infinite $\mc{B}_x$ intersect.
\item[(v)] \begin{enumerate}
                \item[${}^1$] All infinite $\mc{B}_x$ are blocked (above or below)
                \item[${}^2$] All infinite $\mc{B}_x$ equal $\Z^2$
                \end{enumerate}
\item[(vi)] \begin{enumerate}
                 \item[${}^1$] $\exists!$ gigantic $\mc{M}$-component.
                 \item[${}^2$] All $\mc{M}_x=\Z^2$.
                 \item[${}^3$] There are multiple infinite $\mc{M}_x$. 
                 \end{enumerate}
\end{enumerate}
\medskip


\begin{thebibliography}{99}
\bibitem{Arrat}{R.~Arratia, ``Coalescing Brownian motions on the line''. PhD Thesis, Univ. of Wisconsin, Madison (1979).}
\bibitem{BBS}{P.~Balister, B.~Bollob\'as and A.~Stacey, ``Improved upper bounds for the critical probability of oriented percolation in two dimensions''. {\sl Random Structures Algorithms} {\bf 5} (1994), pp. 573--589}
\bibitem{BD11}
N.~Berger and J.-D.~Deuschel.
\newblock A quenched invariance principle for non-elliptic random walk in I.I.D. balanced random environment.
\newblock Preprint, 2011.
\bibitem{DeBE}{K.~De'Bell and J.W.~Essam, ``Estimates of the site percolation probability exponents for some directed lattices''. {\sl J. Phys. A} {\bf 16} (1983), pp. 3145--3147}
\bibitem{DBP}{D.~Dhar, M.~Barma and M.K.~Phani, ``Duality transformations for two-dimensional directed percolation and resistance problems''. {\sl Phys. Rev. Lett.} {\bf47} (1981), pp. 1238--1241}
\bibitem{Durrett}{R.~Durrett, ``Oriented percolation in two dimensions''. {\sl Ann. Probab.} {\bf 12} (1984), pp. 999--1040}

\bibitem{Finch}{S.R.~Finch, {\sl Mathematical constants}. Cambridge University Press, Cambridge (2003)}
\bibitem{GSW}{L.~Gray, R.T.~Smythe, and J.C.~Wierman, ``Lower bounds for the critical probability in percolation models with oriented bonds''. {\sl J. Appl. Probab.} {\bf 17} (1980), pp. 979--986}
\bibitem{Grim99}{G.~Grimmett, {\sl Percolation}. Springer-Verlag, Berlin, 2nd edition (1999)}
\bibitem{Grim02}{G.~Grimmett, ``Infinite paths in randomly oriented lattices''. {\sl Random Structures Algorithms} {\bf 18} (2001), pp. 257--266}
\bibitem{GH02}{G.~Grimmett and P.~Hiemer, ``Directed percolation and random walk''. In {\sl In and out of equilibrium (Mambucaba, 2000)}, ed. V. Sidoravicius, pp. 273--297, Progr. Probab. {\bf 51}, Birkh\"auser, Boston 2002}
\bibitem{HS_RWdRE}{M.~Holmes and T.S.~Salisbury, ``Random walks in degenerate random environments''. Submitted (2012)}
\bibitem{Hughes}{B.D.~Hughes, {\sl Random Walks and Random Environments, Volumes 1, 2}. Oxford University Press, New York (1995/1996)}
\bibitem{JG}{I.~Jensen and A.J.~Guttmann, ``Series expansions of the percolation probability on the directed triangular lattice''. {\sl J. Phys. A} {\bf 29} (1996), pp. 497-517}
\bibitem{J04}{I.~Jensen, ``Improved lower bounds on the connective constants for two-dimensional self-avoiding walks''. {\sl J. Phys. A} {\bf 37} (2004), pp. 11521--11529} 
\bibitem{Lawler}{G.F.~Lawler, ``Low-density expansion for a two-state random walk in a random environment''. {\sl J. Math. Phys.} {\bf 30} (1989), pp. 145--157}
\bibitem{Lig}{T.M.~Liggett, ``Survival of discrete time growth models, with applications to oriented percolation''. {\sl Ann. Probab.} {\bf 5} (1995), pp. 613--636}
\bibitem{Lin}{S.~Linusson, ``A note on correlations in randomly oriented graphs''. Preprint (2009)}
\bibitem{MadS}{N.~Madras and G.~Slade, {\sl The self-avoiding walk}. Birkh\"auser, Boston (1993)}
\bibitem{MarV}{H.O.~M\'artin and J.~Vannimenus, ``Partially directed site percolation on the square and triangular lattices''. {\sl J. Phys. A} {\bf 18} (1985), pp. 1475--1482}
\bibitem{MenP}{M.V.~Men'shikov and K.D.~Pelikh, ``Percolation with several defect types: an estimate of the critical probability for a square lattice''.  {\sl Math. Notes} {\bf 46} (1989), pp. 778-785}
\bibitem{Pete}{G.~Pete, ``Corner percolation on $\Z^2$ and the square root of ${17}$''. {\sl Ann. Probab.} {\bf 36} (2008), pp. 1711--1747}
\bibitem{Redner}{S.~Redner, ``Directionality effects in percolation''. In {\sl The mathematics and physics of disordered media}, ed. B.D. Hughes and B.W. Ninham, pp. 184--200, Lecture Notes in Mathematics {\bf 1035}, Springer 1983}
\bibitem{Russo}{L.~Russo, ``On the critical percolation probabilities''. {\sl Z. Wahrsch. Verw. Gebiete} {\bf 56} (1981), pp. 229--237}
\bibitem{TW}{B.~Toth and  W.~Werner, ``The true self-repelling motion''. {\sl Probab. Th. Rel. Fields} {\bf 111} (1998), pp. 375--452}
\bibitem{Wierman85}{J.C.~Wierman, ``Duality for directed site percolation''. In {\sl Particle systems, random media, and large deviations}, ed. R. Durrett, pp. 363--380, Contemp. Math. {\bf 41}, Amer. Math. Soc. (1985)}
\bibitem{Wierman95}{J.C.~Wierman, ``Substitution method critical probability bounds for the square lattice site percolation model''. {\sl Combin. Probab. Comput.} {\bf 4} (1995), pp. 181--188}
\bibitem{WuZuo}{W.X.~Yuan and Z.~Xinlan, ``On the two-plied oriented percolation on the square lattice and its critical probability function'' (in Chinese). {\sl Acta Math. Appl. Sinica} {\bf 28} (2008), pp. 216--226}
\end{thebibliography}

\end{document}